\documentclass[sn-mathphys]{sn-jnl-id}
\usepackage{latexsym, amsfonts, amsmath, amsthm, amssymb, verbatim, mathrsfs}
\catcode`\|=12\relax  
\usepackage{tikz, float}

\allowdisplaybreaks

\newcommand{\AC}{{AC}}
\newcommand{\BV}{{BV}}

\newcommand{\Pol}{\mathcal{P}_2}
\newcommand{\Lip}{{\mathrm{Lip}}}

\newcommand{\mN}{\mathbb{N}}
\newcommand{\mR}{\mathbb{R}}
\newcommand{\mC}{\mathbb{C}}
\newcommand{\mT}{\mathbb{T}}
\newcommand{\mZ}{\mathbb{Z}}

\newcommand{\bvreal}{BV_\mR(\sigma)}

\newcommand{\abs}[1]{\left\lvert#1\right\rvert}
\newcommand{\diam}[1]{\mathrm{diam}(#1)}

\newcommand{\vecx}{{\boldsymbol{x}}}
\newcommand{\vecy}{{\boldsymbol{y}}}
\newcommand{\vecz}{{\boldsymbol{z}}}

\newcommand{\vecu}{{\boldsymbol{u}}}
\newcommand{\vecv}{{\boldsymbol{v}}}
\newcommand{\vecw}{{\boldsymbol{w}}}

\newcommand{\norm}[1]{\left\lVert#1\right\rVert}
\newcommand{\norminf}[1]{\left\lVert#1\right\rVert_\infty}
\newcommand{\normbv}[1]{\left\lVert#1\right\rVert_{\BV(\sigma)}}

\renewcommand{\Re}{\mathop{\mathrm{Re}}}
\renewcommand{\Im}{\mathop{\mathrm{Im}}}

\newcommand{\ds}{\displaystyle}
\newcommand{\ts}{\textstyle}
\newcommand\st{\thinspace : \thinspace}
\def\ls[#1,#2]{\overline{\vphantom{\vbox to 1.2 ex{}} #1\, #2}}

\def\smod#1{\Bigl \vert #1 \Bigr \vert}

\def\sparen(#1){\Bigl ( #1 \Bigr )}
\def\ssparen(#1){ (#1) }
\newcommand\plist[1]{\bigl[ #1 \bigr]}

\newcommand{\interior}[1]{\mathop{\mathrm{int}}(#1)}

\DeclareMathOperator*{\myvar}{var}
\DeclareMathOperator{\cvar}{\rm cvar}
\DeclareMathOperator{\vf}{vf}

\DeclareMathOperator{\supp}{supp}

\theoremstyle{thmstyleone}%
 \newtheorem{thm}{Theorem}[section]
 \newtheorem{cor}[thm]{Corollary}
 \newtheorem{lem}[thm]{Lemma}
 \newtheorem{prop}[thm]{Proposition}
  \newtheorem{conj}[thm]{Conjecture}

\theoremstyle{thmstylethree}%
 \newtheorem{defn}[thm]{Definition}

\theoremstyle{thmstyletwo}%
 \theoremstyle{remark}
 \newtheorem{rem}[thm]{Remark}
 \newtheorem{ex}[thm]{Example}

 \numberwithin{equation}{section}

\newcommand{\fmax}{f_{\rm max}}
\newcommand{\fmed}{f_{\rm med}}
\newcommand{\fmin}{f_{\rm min}}
\newcommand{\cmax}{c_{\rm max}}
\newcommand{\cmed}{c_{\rm med}}
\newcommand{\cmin}{c_{\rm min}}

\newcommand{\journalname}[1]{\textrm{#1}}

\begin{document}

\title{The Banach algebras $AC(\sigma)$ and $BV(\sigma)$}

\author*[1]{Ian Doust}\email{i.doust@unsw.edu.au}
\author[2]{Michael Leinert}\email{leinert@math.uni-heidelberg.de}
\author[1]{Alan Stoneham}\email{a.stoneham@unsw.edu.au}

\affil*[1]{\orgdiv{School of Mathematics and Statistics},
\orgname{University of New South Wales},
\city{UNSW Sydney} \postcode{2052}, \country{Australia}}

\affil[2]{\orgdiv{Institut f{\"u}r Angewandte Mathematik},
\orgname{Universit{\"a}t Heidelberg},
\orgaddress{Im Neuenheimer Feld 294},
\postcode{D-69120} \city{Heidelberg}
\country{Germany}}


\abstract{
The spaces $BV(\sigma)$ and $AC(\sigma)$ were introduced as part of a program to find a general theory which covers both well-bounded operators and trigonometrically well-bounded operators acting on a Banach space. Since their initial appearance it has become clear that the definitions could be simplified somewhat.
 In this paper we give a self-contained exposition of the main properties of these spaces using this simplified approach.}

\maketitle

\section{Introduction}

The classical definition of the variation of a function of one real variable depends strongly on the order of the real line. Different applications have required extensions of this definition to functions of two or more variables, yielding the challenge of how to replace the order property in a suitable way. Already by the 1930s, Clarkson and Adams \cite{CA} had found quite a collection of mutually inequivalent concepts.

In the 1960s, Smart and Ringrose \cite{Sm,Rin1,Rin2} introduced the theory of well-bounded Banach space operators as an analogue of the theory of self-adjoint operators on a Hilbert space, but now allowing operators whose spectral expansions might only converge conditionally. A Banach space operator $T$ is said to be well-bounded if it admits an $AC[a,b]$ functional calculus for some compact interval $[a,b] \subseteq \mR$. It was shown by Cheng and Doust \cite{CD}, for example, that every compact well-bounded operator $T$ has a representation of the form $T = \sum_{j=1}^\infty \lambda_j P_j$ as a norm convergent (but possibly conditionally convergent) combination of projections onto the eigenspaces of $T$. At least on reflexive Banach spaces, general well-bounded operators admit an integral representation reminiscent of that given by the spectral theorem for self-adjoint operators.
The spectral theorem for well-bounded operators on reflexive Banach spaces allows one to extend the $AC[a,b]$ functional calculus to a $BV[a,b]$ functional calculus, in a similar manner to the way in which the $C(\sigma(T))$ functional calculus for a self-adjoint operator can be extended to the algebra of bounded Borel measurable functions on $\sigma(T)$.

A natural problem then was to provide analogues of normal and unitary operators in this context. After some initial work by Ringrose \cite{Rin2}, steps in this direction were taken by Berkson and Gillespie \cite{BG1,BG2} and Wilson \cite{Wil}. Following Hardy and Krause, Berkson and Gillespie defined a Banach algebra $AC(J\times K)$ of `absolutely continuous' functions defined on a rectangle $J \times K \subseteq \mR^2$ and used this to define the class of $AC$ operators. A little later, they defined a corresponding analogue of unitary operators, the class of trigonometrically well-bounded operators based on an algebra $AC(\mT)$ of functions on the unit circle, and used this in work on operator-valued harmonic analysis (see, for example, \cite{BG3,BGM}). It became apparent however that the more general theory of $AC$ operators had some undesirable properties in terms of spectral theory. For example there exist $AC$ operators $T$ such that  $(1+i)T$ is not an $AC$ operator. It is also more natural (even in the case where $\sigma(T) \subseteq \mR$) to look for a functional calculus model based on functions whose domain is $\sigma(T)$ rather than some potentially much larger set.

To address some of these shortcomings, a new measure of variation of a function defined on a compact subset $\sigma$ of the plane was developed by Ashton \cite{Ash} in 2000. This was used to define the Banach algebras $BV(\sigma)$ and $AC(\sigma)$, the spaces of functions of bounded variation and of absolutely continuous functions respectively, and these algebras were used to define the class of $AC(\sigma)$ operators \cite{AD3}. Importantly, this was consistent with the successful theories for well-bounded and trigonometrically well-bounded operators in the cases that $\sigma \subseteq \mR$ and $\sigma \subseteq \mT$.

Since the original development of the theory in \cite{Ash,AD1}, a number of simplifications of the definitions have been found, and several important new observations concerning these definitions have been made.  A number of gaps in some of the earlier arguments were also identified. There is now much more known about the structure of these spaces (see \cite{AS2,AS3,DA,DL}) and new applications have arisen (see, for example, \cite{NM}).
Unfortunately this provides a challenge for readers, as foundational results are somewhat scattered through the literature and were often shown using different (but equivalent) definitions of variation. The aim of this paper is to give a self-contained development of the spaces $BV(\sigma)$ and $AC(\sigma)$, indicating where there are still some gaps in our knowledge about these spaces.

In Section~\ref{preliminaries} we gather together the basic definitions and properties of the Ashton definition of variation in the plane, and the Banach algebra $BV(\sigma)$. This includes some new results on estimating the variation of a function.

Section~\ref{subalgebras} examines some of the important subalgebras of $BV(\sigma)$, including the polynomials in two (real) variables, the Lipschitz functions and the functions which admit a $C^2$ extension to a neighbourhood of $\sigma$. The absolutely continuous functions are defined as the closure of the set of polynomials in two variables. The final section includes some important properties of absolute continuity in this context. In particular it is shown that absolute continuity can be considered to be a `local property' in the sense that a function is absolutely continuous over the whole set $\sigma$ if and only if it is absolutely continuous on a compact neighbourhood of each point of $\sigma$. (This work is continued in \cite{DLS2} where it is shown, for example, that all functions with a $C^1$ extension to a neighbourhood of $\sigma$ lie in $AC(\sigma)$.)

\section{Basic definitions and properties}\label{preliminaries}

For a complex-valued function $f$ defined on a compact interval $[a,b]$, the concept of its variation is well understood. Given a partition $\mathcal{P}=\{a=x_0<x_1<... <x_n=b\}$ of $[a,b]$, the variation of $f$ over $\mathcal{P}$ is defined to be
\begin{equation*} 
   \myvar(f,\mathcal{P})=\sum_{j=1}^n \abs{ f(x_j)-f(x_{j-1})}.
\end{equation*}
The variation of $f$ is then defined as $\myvar(f)=\ds \sup_{\mathcal{P}}\myvar(f,\mathcal{P})$. One proceeds to define the vector space $BV[a,b]$ of functions of bounded variation over $[a,b]$ as
  \[
     BV[a,b]=\{f:[a,b] \to \mC \, : \, \myvar(f)<\infty\}.
  \]
The space $BV[0,1]$ is a large and rich Banach algebra when endowed with the norm $\|f\|_{BV}=\|f\|_\infty+\myvar(f)$. (Note that throughout, $\norm{f}_\infty$ denotes the supremum rather than the essential supremum of $|f|$.)


Suppose now that $\sigma$ is a nonempty compact subset of the plane and that $f: \sigma \to \mC$. The original definition of the variation of $f$ on $\sigma$ in \cite{Ash} and \cite{AD1} involved calculating the variation of $f$ along suitable curves in the plane. Simpler, but equivalent, definitions were subsequently developed which are in many ways more reminiscent of the one-dimensional definition, and we follow that simplified development here. The equivalence of the different definitions is shown in the appendix to \cite{AS3}. The central concept here is that of the variation factor of a list of points.

For the motivating applications in spectral theory, the set $\sigma$ is the spectrum of a bounded operator, and hence is a subset of $\mC$. The definitions however apply equally well to subsets of $\mR^2$ and so we shall freely swap between the two models of the plane.

The \textbf{curve variation of $f$ on a list $S = \plist{\vecx_0,\vecx_1,\dots,\vecx_n}$} of elements of $\sigma$ will be defined to be
\begin{equation*} 
    \cvar(f, S) =  \sum_{i=1}^{n} \abs{f(\vecx_{i}) - f(\vecx_{i-1})}.
\end{equation*}
Unless $f$ is constant, $\cvar(f,S)$ can be made arbitrarily large, so to arrive at a useful measure of variation, one needs to
introduce a quantity called the variation factor of the list $S$. The properties of this quantity are central to the later theory, so the next subsection is devoted to developing these.

\subsection{Variation factors}

Suppose that $S =
\plist{\vecx_0,\vecx_1,\dots,\vecx_n} = [\vecx_j]_{j=0}^n$ is a finite ordered list of points in the plane. 
Note that the elements of such a list do not need to be distinct. To avoid trivialities, we shall for the moment assume that $n \ge 1$.
Let $\gamma_S$ denote the piecewise linear curve joining the points of $S$ in order.

\begin{defn}\label{defn:x-seg}
Suppose that $\ell$ is a line in the plane. We say that $s_j = \ls[\vecx_j,\vecx_{j+1}]$, the line segment joining $\vecx_j$ to $\vecx_{j+1}$, is a \textbf{crossing segment} of $S$ on $\ell$ if any one of the following hold:
\begin{enumerate}
  \item $\vecx_j$ and $\vecx_{j+1}$ lie on (strictly) opposite sides of $\ell$;
  \item $j=0$ and $\vecx_j \in \ell$;
  \item $\vecx_j \notin \ell$ and $\vecx_{j+1} \in \ell$.
\end{enumerate}
\end{defn}

\begin{ex}
Consider the list $S = [\vecx_0,\vecx_1,\dots,\vecx_8]$ and the line $\ell$ shown in Figure~\ref{x-seg}.

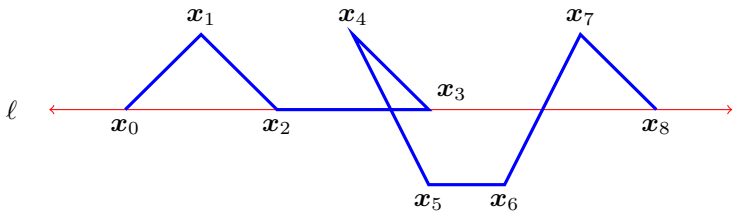
\begin{figure}[h]
\begin{center}
\begin{tikzpicture}
\draw[red,<->] (0,0) -- (9,0);
\draw[blue,very thick] (1,0) -- (2,1) -- (3,0) -- (5,0) -- (4,1) -- (5,-1) --(6,-1) -- (7,1) -- (8,0);
\draw (-0.3,0) node[left] {$\ell$};
\draw (1,0) node[below] {$\vecx_0$};
\draw (2,1) node[above] {$\vecx_1$};
\draw (3,0) node[below] {$\vecx_2$};
\draw (5.3,0) node[above] {$\vecx_3$};
\draw (4,1) node[above] {$\vecx_4$};
\draw (5,-1) node[below] {$\vecx_5$};
\draw (6,-1) node[below] {$\vecx_6$};
\draw (7,1) node[above] {$\vecx_7$};
\draw (8,0) node[below] {$\vecx_8$};
\end{tikzpicture}
\caption{Crossing segments.}\label{x-seg}
\end{center}
\end{figure}

There are five crossing segments of $S$ on $\ell$, namely $\
\ls[\vecx_0,\vecx_1]$ (type 2), $\ls[\vecx_1,\vecx_2]$ (type 3), $\ls[\vecx_4,\vecx_5]$ and $\ls[\vecx_6,\vecx_7]$ (type 1), and $\ls[\vecx_7,\vecx_8]$ (type 3).
\end{ex}
\bigskip

Let $\vf(S,\ell)$ denote the number of crossing segments of $S$ on $\ell$.  An alternative way of calculating $\vf(S,\ell)$ was given in \cite[Appendix]{AS3}. Parameterize  $\gamma_S$ by $[0,1]$ and set $L(S,\ell) = \{t \in [0,1] \st \gamma_S(t) \in \ell\}$. Let $0= t_0 < t_1 < \dots < t_n = 1$ be such that $\gamma_S(t_j) = \vecx_j$ ($j = 0,1,\dots,n$). For each $j$, $\{t \in [t_j,t_{j+1}] \st \gamma_S(t) \in \ell\}$ is either empty, a singleton set, or else is the whole of $[t_j,t_{j+1}]$. Thus, $L(S,\ell)$ can be written as a finite union of disjoint closed connected sets, $L(S,\ell) = \cup_{i=1}^m [b_i,e_i]$. (We allow here that $[b_i,e_i]$ may be a singleton set.)

\begin{prop}\label{vf-component} $\vf(S,\ell)$ is the number of connected components of $L(S,\ell)$.
\end{prop}  

\begin{proof}
With the notation above, let $B = \{b_i\}_{i=1}^m$ be the set of left-hand endpoints of the connected components of $L(S,\ell)$.  Suppose $b \in B$, and that $0 \le j < n$. Then
\begin{enumerate}
\item $b \in (t_j,t_{j+1})$ if and only if $\ls[\vecx_j,\vecx_{j+1}]$ is a crossing segment of type (1).
\item $b = t_0$ if and only if $\ls[\vecx_0,\vecx_{1}]$ is a crossing segment of type (2).
\item $b = t_{j+1}$ if and only if $\ls[\vecx_j,\vecx_{j+1}]$ is a crossing segment of type (3).
\end{enumerate}
Since each interval $[t_j,t_{j+1})$ (and the final point ${t_n}$) can contain at most one element of $B$, this implies that there are as many elements of $B$ (that is, as many connected components of $L(S,\ell)$) as there are crossing segments of $S$ on $\ell$.
\end{proof}

The \textbf{variation factor} of $S$ is defined to be
 \[ \vf(S) = \max_{\ell} \vf(S,\ell). \]
For completeness\footnote{In most later proofs we shall not explicitly consider this possibility. The diligent reader will check that that statements are typically trivially true for lists with just one element.} we include the case $S = [\vecx_0]$ by setting $\vf([\vecx_0],\ell) = 1$ if $\vecx_0 \in \ell$ and zero otherwise, and $\vf([\vecx_0]) = 1$.
Note that $\vf(S)$ is well-defined since $1 \le \vf(S) \le n$. Informally, $\vf(S)$ may be thought of as the maximum number of times any line crosses $\gamma_S$.

It is clear that the variation factor is invariant under rotations or translations of the set $S$, and indeed under any invertible affine transformation of the plane.

It will be important below to see what this means in the case that $S \subseteq \mR$. One can parametrize $\gamma_S$ by arc length, $\gamma_S(t)$, $0 \le t \le T$. If $\ell$ is a horizontal line, then $\vf(S,\ell)$ is either zero or one. Otherwise any line $\ell$ will meet the real axis at a unique point $\vecx$ and $\vf(S,\ell)$ counts the number of solutions to $\gamma_S(t) = \vecx$. Thus $\vf(S)$ is counting the largest number of times $\gamma_S$ visits any single point. This will be needed in the proof of Proposition~\ref{varreal} below.

Given a list of points $S$ and a line $\ell$ in the plane we denote by $S_\ell$ the list comprising the orthogonal projections of these points onto $\ell$.

\begin{prop}\label{vf-proj}
Let $S = \plist{\vecx_0,\vecx_1,\dots,\vecx_n}$ be an ordered list of points and let $\ell$ be a line in the plane. Then $\vf(S_\ell) \le \vf(S)$.
\end{prop}

\begin{proof}
As the variation factor exhibits affine invariance, we may assume that $\ell$ is the real axis. Denote $S_\ell = \plist{\Re(\vecx_0),\dots,\Re(\vecx_n)} = \plist{\vecy_0,\dots,\vecy_n}$, and let $\ell_0$ be a line such that $\vf(S_\ell) =\vf(S_\ell,\ell_0)$. By the earlier remarks, we may assume that $\ell_0$ is orthogonal to the real axis. Suppose that $\ls[\vecy_j,\vecy_{j+1}]$ is a crossing segment of $S_\ell$ on $\ell_0$. Observing that $\vecx_j \in \ell_0$ if and only if $\vecy_j \in \ell_0$, it follows that if $\ls[\vecy_j,\vecy_{j+1}]$ is a crossing segment, then so is $\ls[\vecx_j,\vecx_{j+1}]$. Thus $\vf(S) \ge \vf(S,\ell_0) = \vf(S_\ell,\ell_0) = \vf(S_\ell)$.
\end{proof}

It is worth recording a few simple properties of the variation factor which will be needed later. The first of which shows that if we reverse the order of the points in $S$ to form the list $S'$, the variation factor is unchanged. This follows immediately from Proposition~\ref{vf-component}.

\begin{lem}\label{vf-order}
Let $S=[\vecx_0,\vecx_1,\dots,\vecx_n]$ be a list of points in $\mathbb{C}$ and $S'$ the list obtained by reversing the order of points in $S$ \textup{(\textit{that is $S'=[\vecx_n,\vecx_{n-1},\dots,\vecx_0]$})}. Then for any line $\ell$ in the plane, \textup{vf}$(S,\ell)=\textup{vf}(S',\ell)$, and hence $\vf(S)=\vf(S')$.
\end{lem}

The next result states that if one starts with a list of points $S$ and then removes a point from $S$, one cannot increase the variation factor.

\begin{prop}\label{vf-lemma}
Let $\sigma \subseteq \mC$ be nonempty and compact, let $S$ be an ordered list of elements of $\sigma$, and let $S^{+}$ be a list formed by adding an additional
element into the list at some point. Then for any line in the plane $\vf(S,\ell) \le \vf(S^{+},\ell)$ and hence $\vf(S) \le \vf(S^{+})$. 
\end{prop}

\begin{proof}
Let $S = \plist{\vecx_0,\dots,\vecx_n}$ and let $\ell$ be any line in the plane. Suppose that $\vecw\in\sigma$ is added to $S$ to form $S^+$.
Suppose first that $S^{+} = \plist{\vecw,\vecx_0,\dots,\vecx_n}$. If $\vecx_0 \not\in \ell$, or if $\vecx_0 \in \ell$ and $\vecw \not\in \ell$, then each of the original line segments in $S$ retains its original status as either a crossing or non-crossing segment of $S^{+}$ on $\ell$. If $\vecx_0,\vecw \in \ell$, then $\ls[\vecw, \vecx_0]$  replaces $\ls[\vecx_0,\vecx_1]$ as a crossing segment. In either case, the number of crossing segments is not decreased.

The case that $S^{+} = \plist{\vecx_0,\dots,\vecx_n,\vecw}$ follows from Lemma~\ref{vf-order} combined with the previous case. The remaining case, where the additional point $\vecw$ is added between the $j$ and $(j+1)$th elements of $S$, involves checking a few additional cases.

We consider the four possibilities determined by combinations of $\vecx_j$ and $\vecx_{j+1}$ lying on $\ell$ or not. Suppose that $\vecx_j,\vecx_{j+1}\notin\ell$. Then if $\ls[\vecx_j,\vecx_{j+1}]$ is a crossing segment (necessarily of type (1)), then either $\ls[\vecx_j,\vecw]$ or $\ls[\vecw,\vecx_{j+1}]$ will be a replacement crossing segment. If $\ls[\vecx_j,\vecx_{j+1}]$ is not a crossing segment, then the addition of $\vecw$ will maintain or increase the variation factor.

If $\vecx_j,\vecx_{j+1}\in\ell$, then adding $\vecw$ will maintain or increase the variation factor when $\vecw\in\ell$ or $\vecw\notin\ell$ respectively.

If $\vecx_j\notin\ell$ and $\vecx_{j+1}\in\ell$, then $\ls[\vecx_j,\vecx_{j+1}]$ is a crossing segment of type (3). If $\vecw$ is on the opposite side of $\ell$ to $\vecx_{j}$, then an additional crossing segment is introduced. Otherwise, the variation factor is unchanged. In either case, $\vf(S,\ell)\leq\vf(S^+,\ell)$.

For the final case of $\vecx_j\in\ell$ and $\vecx_{j+1}\notin\ell$, one can reverse the order of $S$. By Lemma~\ref{vf-order}, the variation factor is unchanged, and the problem reduces to the previous case. Thus in all cases, $\vf(S,\ell)\leq\vf(S^+,\ell)$, and hence $\vf(S)\leq\vf(S^+)$.
\end{proof}

Inductively, one deduces that adding any finite number of points to $S$ will either maintain or increase the variation factor.

A final point is that the definition allows that a list $S$ may contain a point $\vecx$ more than once. In the case that $\vecx$ appears twice in succession, the corresponding line segment $\ls[\vecx_j,\vecx_{j+1}]$ is degenerate, and clearly cannot be a crossing segment of type (1) or type (3) on any line. Indeed, it is not hard to check that if $S$ is any list of points, and $S'$ is the sublist formed by removing any point which is equal to its predecessor in $S$, then $\vf(S') = \vf(S)$.

\subsection{Variation in the plane}

Suppose that $\sigma$ is a nonempty compact subset of the plane, that $f: \sigma \to \mC$, and that $S = \plist{\vecx_0,\dots,\vecx_n}$ is a list of elements of $\sigma$.
We define the \textbf{curve variation of $f$ on the list $S$} to be
\begin{equation*} 
    \cvar(f, S) =  \begin{cases}
    \sum_{i=1}^{n} \abs{f(\vecx_{i}) - f(\vecx_{i-1})}
                       & \text{if $n \ge 1$,} \\
    0                 & \text{if $n = 0$.}
       \end{cases}
\end{equation*}
The \textbf{two-dimensional variation} of a function $f : \sigma
\rightarrow \mC$ is defined to be
\begin{equation*}
    \myvar(f, \sigma) = \sup_{S}
        \frac{ \cvar(f, S)}{\vf(S)},
\end{equation*}
where the supremum is taken over all finite ordered lists of elements of $\sigma$.
Letting $\|f\|_\infty = \sup\{|f(z)|\,:\, z\in\sigma\}$, the \textbf{variation norm} is
  \[ \normbv{f} = \norm{f}_\infty + \myvar(f,\sigma) \]
and this is used to define the set of functions of bounded variation on $\sigma$,
  \[ \BV(\sigma) = \{ f: \sigma \to \mC \st \normbv{f} < \infty\}. \]

Our first task is to show that $\normbv{\cdot}$ is indeed a norm, and that $\BV(\sigma)$ is always a Banach algebra. Perhaps, just as importantly, we shall shortly show that $\BV(\sigma)$ always contains all sufficiently nice functions, and in the case that $\sigma = [a,b] \subseteq \mR$, the above definition just gives the usual space met in real analysis.

We begin with a simple lemma which just requires straightforward uses of the triangle inequality.

\begin{lem} \label{lem:bvc:5540}
Let $\sigma \subseteq \mC$ be a nonempty compact set and suppose that  $f, g : \sigma \rightarrow \mC$. Then
\begin{enumerate}
    \item $\myvar(f + g, \sigma) \leq \myvar(f, \sigma) + \myvar(g, \sigma)$,
    \item $\myvar(f g, \sigma) \leq \norminf{f} \myvar(g, \sigma) + \norminf{g} \myvar(f, \sigma)$,
    \item $\myvar(\alpha f, \sigma) = \abs{\alpha} \myvar(f, \sigma)$,
    \item $|\myvar(f,\sigma) - \myvar(g,\sigma)| \le \myvar(f-g,\sigma)$,
    \item $\myvar(|f|,\sigma) \leq \myvar(f,\sigma)$.
\end{enumerate}
\end{lem}

\begin{proof} The first, third and fifth statements are straightforward.
For the second statement, we have
\begin{align*}
\myvar(fg,\sigma)	&	=\sup_{S} \frac{\cvar(fg,S)}{\vf(S)} \\
&	=\sup_S \frac{1}{\vf(S)} \sum_{i=1}^{n}|f(\vecx_i)g(\vecx_i)-f(\vecx_{i-1})g(\vecx_{i-1})| \\
&	\leq \sup_S\frac{1}{\vf(S)}\sum_{i=1}^{n}\big(|f(\vecx_i)||g(\vecx_i)-g(\vecx_{i-1})|+|f(\vecx_i)-f(\vecx_{i-1})||g(\vecx_{i-1})| \big) \\
&	\leq \sup_S \frac{\|f\|_{\infty} \cvar(g,S)}{\vf(S)}+\sup_S\frac{\|g\|_{\infty} \cvar(f,S)}{\vf(S)} \\
&	=\|f\|_\infty \myvar(g,\sigma)+\|g\|_\infty\myvar(f,\sigma).
\end{align*}
The fourth statement follows by writing $f = (f-g)+g$ and using (1) and (3).
\end{proof}

A simple consequence of the above lemma is that
 $BV(\sigma)$ is indeed a normed algebra.

 \begin{lem} \label{lem:bvc2}
Let $\sigma \subseteq \mC$ be nonempty compact. Let $f, g : \sigma
\rightarrow \mC$. Then
\begin{enumerate}
    \item $\normbv{f+g} \leq \normbv{f} + \normbv{g}$,
    \item $\normbv{f g} \leq \normbv{f} \normbv{g}$,
    \item $\normbv{\alpha f} = \abs{\alpha} \normbv{f}$.
\end{enumerate}
\end{lem}

%
It is easy to check that $f \in BV(\sigma)$ if and only if both $\Re f$ and $\Im f$ lie in $BV(\sigma)$. We shall denote the real normed algebra of real-valued functions of bounded variation by $\bvreal$.

Before continuing it is worth confirming that unless $\sigma$ is a singleton set, $BV(\sigma)$ always contains more than just the constant functions. Given $U \subseteq \sigma$, let $\chi_U$ denote the characteristic function of $U$ (as a function from $\sigma$ to $\mC$).

\begin{lem}\label{char-fns}
Suppose that $\sigma \subseteq \mC$ is a nonempty compact set. Then $\chi_{\{\vecz\}} \in BV(\sigma)$ for all $\vecz \in \sigma$.
\end{lem}

\begin{proof} The result is clear if $\sigma$ is a singleton set, so assume that $\sigma$ has at least two elements.
Fix $\vecz \in \sigma$ and let $f = \chi_{\{\vecz\}}$. Let $S = [\vecx_0,\dots,\vecx_n]$ be a list of points in $\sigma$ for which $\cvar(f,S) \ne 0$. Since, by Proposition~\ref{vf-lemma}, omitting points in $\sigma$ cannot increase the variation factor, we may assume that $|f(\vecx_j)-f(\vecx_{j+1})| = 1$ for all $j$, and so $\cvar(f,S) = n$. Now choose a line $\ell$ which passes through $\vecz$ but none of the other elements of $S$. Note that if a segment $\ls[\vecx_{j-1},\vecx_j]$ is not a crossing segment of $S$ on $\ell$, then $\ls[\vecx_j,\vecx_{j+1}]$ must be and hence $\vf(S) \ge \vf(S,\ell) \ge n/2$. It follows that $\cvar(f,S)/\vf(S) \le 2$ and hence $\myvar(f,\sigma) \le 2$.
\end{proof}

Easy consequences of this lemma are recorded in the following proposition.

\begin{prop}\label{bv-notc}
Suppose that $\sigma$ is an infinite compact subset of $\mC$. Then
\begin{enumerate}
\item $BV(\sigma)$ is infinite dimensional,
\item there exists a discontinuous function in $BV(\sigma)$.
\end{enumerate}
\end{prop}

Of course, if $\sigma$ is a finite set, then every function $f: \sigma \to \mC$ is of bounded variation.

\subsection{$BV(\sigma)$ is a Banach algebra}

\begin{thm}
 Let $\sigma \subseteq \mC$ be nonempty and compact. Then $BV(\sigma)$ is a Banach algebra.
\end{thm}

\begin{proof}
Lemma 2.7 shows that $BV(\sigma)$ is a normed algebra, so it remains to show that $BV(\sigma)$ is complete.

Suppose then that $(f_k)_{k=1}^\infty$ is a Cauchy sequence in $BV(\sigma)$. By the definition of the variation norm, this sequence certainly converges uniformly to a function $f$.

Fix $\varepsilon > 0$. Choose $N$ such that for all $j,k \ge N$, $\normbv{f_j - f_k} < \frac{\varepsilon}{2}$. Suppose now that $S = [\vecx_0,\dots,\vecx_n]$ is a list of points in $\sigma$. As $f_k \to f$ uniformly we can choose $N_S \ge N$ such that $\max\limits_{0\leq i\leq n} |f(\vecx_i)-f_{N_S}(\vecx_i)| < \frac{\varepsilon}{4n}$. Then whenever $k\ge N$ 
  \begin{align*}
   \cvar(f - f_k  , S)  &= \cvar(f - f_{N_S}+f_{N_S}-f_k,S) \\
        & \le  \cvar(f - f_{N_S},S) + \cvar(f_{N_S}-f_k,S) \\
        &= \sum_{i=1}^n | (f(\vecx_i)-f_{N_S}(\vecx_i))  -  (f(\vecx_{i-1})-f_{N_S}(\vecx_{i-1}))| \\
        & \hskip 1cm  +   \cvar(f_{N_S}-f_k,S) \\
        &\le \sum_{i=1}^n | f(\vecx_i)-f_{N_S}(\vecx_i)| + \sum_{i=1}^n |f(\vecx_{i-1})-f_{N_S}(\vecx_{i-1})|\\
                &\hskip 1cm  +  \cvar(f_{N_S}-f_k,S) \\ 
         &<  \frac{\varepsilon}{2} + \cvar(f_{N_S} - f_k,S).
  \end{align*}
Recalling that $\vf(S) \ge 1$,
  \[ \frac{\cvar(f - f_k, S)}{\vf(S)} < \frac{\varepsilon}{2\vf(S)} + \frac{\cvar(f_{N_S} - f_k,S)}{\vf(S)}
   < \frac{\varepsilon}{2} + \norm{f_{N_S}-f_k}_{BV(\sigma)} < \varepsilon
  \]
and so  $\myvar(f-f_k,\sigma) \le \varepsilon$.
Hence $f\in BV(\sigma)$ and $f_k \to f$ in $BV(\sigma)$.
\end{proof}

One of the `design criteria' for the definition of variation in the plane is that it should be invariant under translations and rotations. Since all of the quantities involved in the definition of variation are invariant under affine transformations of $\sigma$ it is easy to see that this is the case.

\begin{thm}\label{aff-inv}
 Suppose that $\alpha,\beta \in \mC$, where $\alpha \ne 0$, and let $\phi(z) = \alpha z + \beta$.
 Suppose that $\sigma$ is a nonempty compact subset of $\mC$, and let ${\hat \sigma} = \phi(\sigma)$.
 For $f: \sigma \to \mC$ let ${\hat f}: {\hat \sigma} \to \mC$ be defined by ${\hat f} = f \circ \phi^{-1}$.
 Then the map $f\to {\hat f}$ is an isometric Banach algebra isomorphism from $BV(\sigma)$ to $BV({\hat \sigma})$.
\end{thm}

Affine transformations of $\mR^2$ are of course slightly more general, and one could just as easily rewrite the theorem in that setting.

Another more or less immediate consequence of the definition is that restrictions of functions of bounded variation are also of bounded variation.

\begin{thm}\label{bv-restrict}
Suppose that $\sigma_1,\sigma$ are nonempty compact subsets of the plane with $\sigma_1 \subseteq \sigma$. If $f \in BV(\sigma)$ then $f|\sigma_1 \in BV(\sigma_1)$, with $\norm{f|\sigma_1}_{BV(\sigma_1)} \le \norm{f}_{BV(\sigma)}$.
\end{thm}

\begin{proof}
Suppose that $\sigma,\sigma_1$ are nonempty compact subsets of $\mC$, $\sigma_1\subseteq\sigma$and that $f\in BV(\sigma)$. Since the set of finite lists of points in $\sigma_1$ is contained in the set of finite lists of points in $\sigma$, we obtain that $\myvar(f|\sigma_1,\sigma_1)\leq\myvar(f,\sigma)$, and the result follows since $\|f|\sigma_1\|_\infty\leq\|f\|_\infty$.
\end{proof}

\subsection{The case $\sigma \subseteq \mR$}

The space of functions of bounded variation on an interval $[a,b]$ is one of the classical spaces of real analysis. Indeed there is no great complication in considering more general subsets of the real line and this was already considered in Saks' classic 1930's text \cite{S}. In order to justify our use of the term `bounded variation' for functions defined on more general subsets of the plane, we should at this point check that in the case that $\sigma \subseteq \mR$, the definition above does just reduce to the classical one.

Suppose then that $\sigma$ is a nonempty compact subset of $\mR$ and that $f: \sigma \to \mC$. Suppose that $\mathcal{P} = \{\vecx_0 < \vecx_1 < \dots < \vecx_n\} \subseteq \sigma$ is a partition. Any such partition has a corresponding list $S = [\vecx_0,\vecx_1,\dots,\vecx_n]$ which has variation factor $1$, so it is immediate from the classical definition of variation that $\myvar(f) \le \myvar(f,\sigma)$.


\begin{prop}\label{varreal}
Suppose that $\sigma$ and $f$ are as above. Then $\myvar(f) = \myvar(f,\sigma)$.
\end{prop}

\begin{proof}
From the remark above, we only need to show that
$\myvar(f) \geq \myvar(f,\sigma)$.
Suppose that $S = \plist{\vecx_0,\dots,\vecx_n}$ is a list of points in $\sigma$. The list $S$ generates a partition $\mathcal{P} = \{t_0,\dots,t_m\}$ by ordering the distinct elements of $S$. Now form a  list ${\hat S}$ by starting with the list $S$ and, for each $j$, inserting between $\vecx_j$ and $\vecx_{j+1}$ all of the points of $\mathcal{P}$ that are passed on the line segment $s_j = \ls[\vecx_j,\vecx_{j+1}]$. (By way of an example, if $S = [2,4,1,1,3,2]$ then $\mathcal{P} = \{1,2,3,4\}$ and ${\hat S} = [2,3,4,3,2,1,1,2,3,2]$.)

By the triangle inequality,
  \[ \cvar(f,S) = \sum_{j=1}^n |f(\vecx_j) - f(\vecx_{j-1})|
        \le \cvar(f,{\hat S})
        = \sum_{k=1}^m c_k |f(t_k) - f(t_{k-1})|
  \]
where $c_k$ counts the number of times the term $|f(t_k) - f(t_{k-1})|$ appears in the sum for $\cvar(f,{\hat S})$. (In the example $(c_1,c_2,c_3) = (2,4,2)$.) Recall that $\gamma_S$ denotes the piecewise linear curve joining the points of $S$ in order. Note that the number of times that $\gamma_S$ crosses any point in the interior of the line segment $\ls[t_{k-1},t_k]$ is  $c_k$, and so $\vf(S) \ge \max\limits_kc_k$. It follows that
  \[ \cvar(f,S) \le \vf(S) \sum_{k=1}^m |f(t_k) - f(t_{k-1})| \le \vf(S) \cdot \myvar f \]
and hence that
  \[ \myvar(f,\sigma) \le \myvar(f) \]
which completes the proof.
\end{proof}

\begin{cor}
If $\sigma \subseteq \mR$ is nonempty and compact, then $BV(\sigma)$ is the classical space of functions of bounded variation on $\sigma$.
\end{cor}

By the affine invariance of these spaces (Theorem~\ref{aff-inv}) essentially the same is true whenever $\sigma$ is contained in any line in the plane.

\subsection{Functions constant on lines}

Calculating $\myvar(f,\sigma)$ for a function $f$ is generally quite challenging. One class of functions which is easier to deal with are the functions which only vary in one direction. In this section we shall make that idea precise and then use this fact to show that all polynomials in two real variables are of bounded variation.

Suppose that $\sigma$ is a nonempty compact subset of $\mC$. Let $\sigma_\mR$ denote the projection of $\sigma$ on to the real axis and suppose that $u: \sigma_\mR \to \mC$. The function $u$ can be extended to all of $\sigma$ by setting
  \[ {\hat u}(x+iy) = u(x), \qquad x+iy \in \sigma, \]
so that $\hat u$ is constant on all vertical lines.

\begin{thm}\label{ext-1var}
Suppose that $\sigma$ and $u$ are as above. If $u \in BV(\sigma_\mR)$ then ${\hat u} \in BV(\sigma)$ and $\normbv{{\hat u}} \le \norm{u}_{BV(\sigma_\mR)}$.
\end{thm}

\begin{proof}
Suppose that $S = \plist{\vecx_0,\dots,\vecx_n}$ is a list of points in $\sigma$ and let $S_\mR = \plist{x_0,\dots,x_n}$ be the list of the  real parts of these points (in the same order). By Proposition~\ref{vf-proj}, $\vf(S_\mR) \le \vf(S)$ and so
  \[ \frac{\cvar({\hat u},S)}{\vf(S)} = \frac{\cvar(u,S_\mR)}{\vf(S)} \le \frac{\cvar(u,S_\mR)}{\vf(S_\mR)} \le \myvar(u,\sigma_\mR).\]
Taking the supremum over all lists $S$ and adding $\norminf{\hat u} = \norminf{u}$ completes the proof.
\end{proof}

Again, by the affine invariance of the variation, one can easily extend this result by replacing the real line with any other line $L$ in the plane.

It is worth noting that in the above theorem, one does not get equality in general.

\begin{ex}
Let $\sigma = \{(-1,1),(0,0),(1,1)\}$, so $\sigma_\mR = \{-1,0,1\}$, and $u(x) =|x|$. Then we have that $\myvar(u,\sigma_{\mR}) = 2$ but $\myvar({\hat u},\sigma) = 1$.

To see this, note that taking $S = [(-1,1),(0,0)]$ shows that $\myvar({\hat u},\sigma) \ge 1$. Also, we could consider ${\hat u}$ to be the extension function ${\hat u}(x+iy) = v(y)$ where $v \in BV(\{0,1\})$, $v(y) = y$, and so by the theorem, $\myvar({\hat u},\sigma) \le \myvar(v,\{0,1\}) = 1$.
\end{ex}

\subsection{The case $\sigma=\mT$}

In \cite{BG2}, Berkson and Gillespie defined a notion of variation for functions defined on the complex unit circle $\mT$. It was a rather natural definition in the sense that it used the standard parametrisation of the circle to represent a function defined on $\mT$ as a function defined on $[0,2\pi]$. To be explicit, for a function $f:\mT\to\mC$, its variation in the sense of Berkson and Gillespie is defined to be $\myvar_{\mT}(f)=~\myvar_{[0,2\pi]}f\big(e^{i(\cdot)}\big)$. The set of all functions on the circle satisfying $\myvar_{\mT}(f)<\infty$ forms a Banach algebra when endowed with the norm $\|f\|=\|f\|_{\infty}+\myvar_{\mT}(f)$. We shall denote this algebra by $BV_{BG}(\mT)$. This definition was used then to define the class of trigonometrically well-bounded operators. As part of the aim of introducing our current definition was to generalize the theories of well-bounded and trigonometrically well-bounded operators, it was important that in fact $BV(\mT)=BV_{BG}(\mT)$. As was shown in \cite{AD2}, the two algebras are isomorphic, but not isometrically. Instead, we obtain an equivalence of norms. The proof in \cite{AD2} was given using the original definition of variation, so we give here a proof in line with the current development.

\begin{lem}\label{bv-T1}
For all $f\in BV(\mT)$, $\myvar_{\mT}(f)\leq 2\myvar(f,\mT)$.
\end{lem}

\begin{proof}
Suppose that
$f\in BV(\mT)$ and $\{0=\theta_0 < \theta_1 < \dots < \theta_n=2\pi\}$
is a partition of $[0,2\pi]$ with $n \ge 2$. For $N\in\mZ^+$, define $S_N$ to be the list of points obtained by starting at $e^{i\theta_0}$ and then repeating $N$ times the sublist $S=[e^{i\theta_1},\dots , e^{i\theta_n}]$. Let $\ell$ be the line which passes through the circle at $e^{i\theta_0}$ and the midpoint on the circle between $e^{i\theta_1}$ and $e^{i\theta_2}$. Then $\vf(S_N) \ge \vf(S_N,\ell) = 2N+1$.
Hence
\begin{align*}
\myvar(f,\mT) &	= \sup_S\frac{\cvar(f,S)}{\vf(S)}\\
&	 \geq \sup_N \frac{\cvar(f,S_N)}{\vf(S_N)}\\
&	 = \sup_N \frac{N}{2N+1} \sum_{j=1}^{n} |f(e^{i\theta_j})-f(e^{i\theta_{j-1}})|,
\end{align*}
from which it follows that $\myvar_{\mT}(f)\leq 2\myvar(f,\mT)$.
\end{proof}

It follows that $BV(\mT)\subseteq BV_{BG}(\mT)$. Establishing that the reverse inclusion holds requires some work. Before that, it is worth noting that the bound given in Lemma~\ref{bv-T1} is sharp, since if $f$ is taken to be the characteristic function of $\{1\}$, then $\myvar_{\mT}(f)=2$, and $\myvar(f,\mT)$ can be calculated to be 1 as follows. By taking a simple list, such as $S=[1,i]$, one sees that $\myvar(f,\mT)\geq1$. For the reverse inequality, take an arbitrary finite list $S=[e^{i\theta_0}, \dots ,e^{i\theta_n}]$ of points in $\mT$. Moreover, we can assume that $e^{i\theta_j}=1$ for some $j$. Then
\begin{align*}
\cvar(f,S)=\big| \{j  \,:\, |f(e^{i\theta_j})-f(e^{i\theta_{j-1}})|=1\} \big|,
\end{align*}
that is the number of times the piecewise linear curve $\gamma_S$ leaves or returns to 1. Consider any vertical line $\ell$ such that the only point of $S$ lying to the right of $\ell$ is 1 and all other points of $S$ lie to the left of $\ell$. Then one finds that $\cvar(f,S)=\vf(S,\ell)\leq\vf(S)$. This establishes the reverse inequality $\myvar(f,\mT) \leq 1$, and consequently $\myvar_{\mT}(f) = 2\myvar(f,\mT)$.

We now move on to the task of showing $BV_{BG}(\mT)\subseteq BV(\mT)$. We shall say that any continuous, orientation preserving bijection $\tau:\mT\to\mT$ is a \textbf{reparametrisation} of~$\mT$. For any $f:\mT\to\mC$, let $f_\tau=f\circ\tau$.

\begin{lem}\label{var-reparam}
Suppose that $f:\mT\to\mC$ and that $\tau$ is a reparametrisation of $\mT$. For a list of points $S=[\vecz_0,\dots,\vecz_n]$, let $S_\tau=\tau(S)=[\tau(\vecz_0),\dots,\tau(\vecz_n)]$. Then
\begin{enumerate}
\item $\myvar_{\mT}(f)=\myvar_{\mT}(f_\tau)$;
\item $\vf(S)=\vf(S_\tau)$;
\item $\myvar(f,S)=\myvar(f_\tau,S_\tau)$.
\end{enumerate}
\end{lem}

\begin{proof}
It should be clear that (1) holds. For (2), suppose that $\ell$ is a line in $\mC$ such that $\vf(S,\ell)=\vf(S)$. The cases where $\ell$ is disjoint from $\mT$ or where $\ell$ only touches $\mT$ at a single point are easily dealt with, so we will assume that $\ell$ intersects $\mT$ at two distinct points, say $\vecx$ and $\vecy$. Define $\ell_\tau$ to be the line passing through $\tau(\vecx)$ and $\tau(\vecy)$. Since $\mT$ is strictly convex, it is then simple to check that if $\ls[ \vecz_j,\vecz_{j+1}]$ is a crossing segment of $S$ on $\ell$, then $\ls[ \tau(\vecz_j),\tau(\vecz_{j+1})]$ is a crossing segment of $S_\tau$ on $\ell_\tau$ of the same type. Thus $\vf(S)\leq\vf(S_\tau)$. A similar argument establishes the reverse inequality, and so $\vf(S_\tau) = \vf(S)$. Since $\cvar(f,S)=\cvar(f_\tau,S_\tau)$, this also proves (3).
\end{proof}

\begin{thm}\label{bv-T2}
For all $f\in BV_{BG}(\mT)$, $\myvar(f,\mT)\leq\myvar_{\mT}(f)$.
\end{thm}

\begin{proof}
Fix a list of points $S=[\vecz_0,\dots,\vecz_n]$ in $\mT$ and choose a paramterisation $\tau$ such that $\text{Im}(\vecz_j)>0$ for each $j$. By Lemma~\ref{var-reparam}, we have that $\myvar(f,S)=\myvar(f_\tau,S_\tau)$.
The map $h(x) = x+i\sqrt{1-x^2}$, $x \in [-1,1]$ has range $\mT^+$, the upper half of $\mT$.  Define $\hat{f_\tau}:[-1,1] \to \mC$ by $\hat{f_\tau}(x)=f_\tau(h(x))$. By Proposition~\ref{varreal} and Theorem~\ref{ext-1var}, we have that
\begin{align*}
\frac{\cvar(f,S)}{\vf(S)}
  = \frac{\cvar(f_\tau,S_\tau)}{\vf(S_\tau)}
  \leq \frac{\cvar\bigl(\hat{f_\tau},(S_\tau)_{\mR}\bigr)}{\vf\big((S_\tau)_{\mR}\big)} \leq \myvar_{[-1,1]}\bigl(\hat{f_\tau}\bigr).
\end{align*}
Noting that
\begin{align*}
\myvar_{[-1,1]}\bigl(\hat{f_\tau}\bigr)
   =\myvar_{[0,\pi]}(f_\tau(e^{i(\cdot)}))
   \leq {\ts \myvar_\mT(f_\tau)}
   = {\ts \myvar_\mT(f)},
\end{align*}
on taking the supremum over all lists $S$ we conclude that $\myvar(f,\mT) \leq \myvar_\mT(f)$.
\end{proof}

Combining Lemma~\ref{bv-T1} with Theorem~\ref{bv-T2} yields the following corollary.

\begin{cor}\label{bg-bv}
$BV(\mT)$ is isomorphic to $BV_{BG}(\mT)$ as Banach algebras.
\end{cor}

\begin{thm}\label{conv-curve}
If $\sigma$ is a strictly convex closed curve, then $BV(\sigma)$ is isometrically isomorphic to $BV(\mT)$ as Banach algebras.
\end{thm}

\begin{proof}
By affine invariance, we may assume that the origin lies within the interior of $\sigma$. As $\sigma$ is a convex closed curve, it can be parametrized by the angle from the positive real axis: $\gamma_\sigma:[0,2\pi]\to\mC$ where $\gamma_\sigma(\theta)=r(\theta)e^{i\theta}$ for some continuous $r:[0,2\pi]\to[0,\infty)$. We can then define a mapping $\Phi:BV(\mT)\to BV(\sigma)$ by $\Phi(f)(\gamma_\sigma(\theta))=f(e^{i\theta})$. This mapping is an isomorphism of algebras, and is an isometry. To see this, we note that for any finite list of points $S=[e^{i\theta_0},\dots,e^{i\theta_n}]$ of $\mT$, write $\hat{S}=[r(\theta_0)e^{i\theta_0},\dots,r(\theta_n)e^{i\theta_0}]$. Then it should be clear that $\cvar(f,S)=\cvar(\Phi(f),\hat{S})$. Now let $\ell$ be a line such that $\vf(S)=\vf(S,\ell)$. Without loss, $\ell$ intersects $\mT$ at two distinct points, say $e^{i\alpha}$ and $e^{i\beta}$. Let $\ell_\sigma$ be the line joining $r(\alpha)e^{i\alpha}$ and $r(\beta)e^{i\beta}$. Then it can be checked that due to the strict convexity of $\mT$ and $\sigma$ that every crossing segment of $\ell$ on $S$ corresponds to a crossing segment the same type of $\ell_\sigma$ on $\hat{S}$. Hence $\vf(S)\leq\vf(\hat{S},\ell_\sigma)\leq\vf(\hat{S})$. Similarly, $\vf(\hat{S})\leq\vf(S)$. Hence $\myvar(f,S)=\myvar(\Phi(f),\sigma)$ and $\|f\|_{BV(\mT)}=\|\phi(f)\|_{BV(\sigma)}$.
\end{proof}

\subsection{Finite sets}

Even in the case that the set $\sigma$ is finite, the definition of the variation of a function involves a supremum over an infinite family of lists of elements of the set. Of course when $\sigma$ is a finite subset of $\mR$ this supremum is achieved by taking the list $S$ to contain all the points of $\sigma$ in order, and hence is easily calculated.
 It would be desirable to have a `finitely computable' formula which applied for more general finite sets. In general, no such formula is currently known, but some information is available in special cases.

If we say that $\sigma = \{\vecz_i\}_{i=1}^m$ consists of the vertices of a convex polygon $C_\sigma$ we shall assume that the points $\vecz_1,\dots,\vecz_m$ are listed so that they are consecutive vertices if the boundary of $C_\sigma$ is followed in an anticlockwise manner.

\begin{thm}\label{finite-convex}
Suppose that $\sigma = \{\vecz_i\}_{i=1}^m$ and ${\hat \sigma} = \{\vecw_i\}_{i=1}^m$ each consist of the vertices of convex polygons $C_\sigma$ and $C_{\hat \sigma}$. Define $h: \sigma \to {\hat \sigma}$ by $h(\vecz_i) = \vecw_i$, $i=1,\dots,m$, and define $\Phi: BV(\sigma) \to BV({\hat \sigma})$ by $\Phi(f) = f \circ h^{-1}$. Then $\Phi$ is an isometric Banach algebra isomorphism.
\end{thm}

\begin{proof}
The proof is almost identical to that of Theorem~\ref{conv-curve}. Again it is clear that $\Phi$ is an algebra isomorphism. Given $f \in BV(\sigma)$, a list $S = [\vecx_0,\dots,\vecx_n]$ of points in $\sigma$, and a corresponding set of points ${\hat S} = [h(\vecx_0),\dots,h(\vecx_n)]$, then as in that proof, one can show that $\vf(S) = \vf({\hat S})$ and $\cvar(f,S) = \cvar(\Phi(f),{\hat S})$. This is enough to show that $\myvar(f,\sigma) = \myvar(\Phi(f),{\hat \sigma})$ and hence that $\norm{f}_{BV(\sigma)} = \norm{\Phi(f)}_{BV(\hat \sigma)}$.
\end{proof}

A consequence of this theorem is that for finite sets which comprise the vertices of a convex set, one may, for the purposes of calculating the variation, assume that the points lie on the unit circle.

\begin{thm}\label{conv-bounds-thm}
Suppose that $\sigma = \{\vecz_i\}_{i=1}^m$ consists of the vertices of a convex polygon, and that $f:\sigma \to \mC$. Then
  \begin{equation}\label{convex-bounds}
  \frac{1}{2} \Bigl(|f(\vecz_m) - f(\vecz_1)|
     + \sum_{i=1}^{m-1} |f(\vecz_i) - f(\vecz_{i+1})| \Bigr)
       \le \myvar(f,\sigma)
       \le \sum_{i=1}^{m-1} |f(\vecz_i) - f(\vecz_{i+1})|.
   \end{equation}
\end{thm}

\begin{proof}
The left hand inequality follows exactly as in the proof of Lemma~\ref{bv-T1}. For the right hand inequality, we can use Theorem~\ref{finite-convex} to move the points into the upper half of the unit circle so that $\vecz_j = e^{i \theta_j}$ with $0 \le \theta_1 < \theta_2 < \dots < \theta_m \le \pi$. Then $\sigma_\mR = \{\cos \theta_j\}_{j=1}^m$ and, if one sets $g(\theta_j) = f(\vecz_j)$, $j=1,\dots,m$, then by Theorem~\ref{ext-1var}, $\myvar(f,\sigma) \le \myvar(g,\sigma_\mR)$. But $\myvar(g,\sigma_\mR)$ is just the expression on the right-hand side of (\ref{convex-bounds}).
\end{proof}

In the case that $\sigma = \{\vecz_j\}_{j=1}^3$ comprises three non-collinear points we can say a little more.

\begin{thm}\label{var-triangles}
	Suppose $\sigma=\{\vecz_1,\vecz_2,\vecz_3\}$ is a set of three non-collinear points in $\mC$ and that $f: \sigma \to \mC$. Then
	\begin{align*}
		\myvar(f,\sigma) = \frac{1}{2}\sum_{i<j} |f(\vecz_i)-f(\vecz_j)|.
	\end{align*}
\end{thm}

\begin{proof} Theorem~\ref{conv-bounds-thm} for $m=3$ says that the expression given is a lower bound for $\myvar(f,\sigma)$, so it just remains to show that it is also an upper bound.

Let $S = [\vecx_0,\dots,\vecx_n]$ be a list of points in $\sigma$. (We may safely assume that the list does not contain the same point in consecutive positions in the list, and that $n \ge 1$.) Each term in $\cvar(f,S)$ is of the form $f_{ij} := |f(\vecz_i) - f(\vecz_j)|$ with $i < j$. Let $c_{ij}$ denote the number of times that $f_{ij}$ occurs in the sum for $\cvar(f,S)$. Let $\fmax,\fmed,\fmin$ be the $f_{ij}$'s arranged in non-increasing order, and let $\cmax$, $\cmed$ and $\cmin$ be the $c_{ij}$'s arranged in non-increasing order.

Using the rearrangement inequality, and noting that $\vf(S) \geq \cmax+\cmed \geq 1$, we have that
	\begin{align*}
		\frac{\cvar(f,S)}{\vf(S)} \leq \frac{\cmax\fmax + \cmed\fmed + \cmin\fmin}{\cmax + \cmed}.
	\end{align*}
The triangle inequality implies that $\fmed+\fmin \geq \fmax$ and hence that
  \[ 0  \leq (\cmax-\cmed) (\fmed +\fmin - \fmax)  + 2\fmin(\cmed-\cmin). \]
This can be rearranged to
  \[ 2(\cmax\fmax + \cmed\fmed + \cmin\fmin)
            \leq (\cmax+\cmed)(\fmax+\fmed+\fmin) \]
or equivalently,
  \[ \frac{\cmax\fmax + \cmed\fmed + \cmin\fmin}{\cmax + \cmed}
         \leq \frac{\fmax+\fmed+\fmin}{2}. \]
	Thus taking the supremum over all such $S$ will yield
  \begin{align*}
	\myvar(f,\sigma)
        \leq \frac{1}{2} \left( \fmax + \fmed + \fmin \right)
        = \frac{1}{2} \sum_{i < j} |f(\vecz_i)-f(\vecz_j)|.
	\end{align*}
\end{proof}

The requirement that the points in $\sigma$ are non-collinear in Theorem~\ref{var-triangles} is essential.

It is currently an open question as to whether $\myvar(f,\sigma)$ is always equal to the lower bound in (\ref{convex-bounds}) if $\sigma$ has more than 3 points. We note that this is the case if $f$ is real valued with $f(\vecz_1) \leq f(\vecz_2) \leq \dots \leq f(\vecz_m)$, since then the upper and lower bounds coincide.


\subsection{Lattice properties of $\bvreal$}

Here we establish some additional properties of $BV(\sigma)$, particularly that the space $BV_\mR(\sigma)$ of real-valued functions is a lattice under pointwise operations. We first note that $BV_\mR(\sigma)$ is closed under taking absolute values since, as stated in Lemma~\ref{lem:bvc:5540}, $\myvar(|f|,\sigma)\leq \myvar(f,\sigma)$.

\begin{lem}\label{absBV}
	If $f \in BV(\sigma)$, then $\norm{\,|f|\,}_{BV(\sigma)}  \leq \norm{f}_{BV(\sigma)}$ and so $|f| \in BV(\sigma)$.
\end{lem}

Given $f,g:\sigma\to\mR$, let $f\vee g$ and $f\wedge g$ denote the pointwise maximum and minimum of $f$ and $g$ respectively.

\begin{prop}
	If $f,g\in BV_\mR(\sigma)$, then $f\vee g$ and $f\wedge g$ are in $BV(\sigma)$ with
	\begin{align*}
		\normbv{f\vee g},\normbv{f\wedge g}\leq \normbv{f}+\normbv{g}.
	\end{align*}
\end{prop}

\begin{proof}
Suppose that $f,g\in BV(\sigma)$ are real valued. Noting that
\begin{align*}
f\vee g= \frac{1}{2}\big( f + g + |f-g| \big),
\end{align*}
and that $|f-g|\in BV(\sigma)$ by Lemma~\ref{absBV}, we find that $f\vee g\in BV(\sigma)$. Similarly, using $f\wedge g=\frac{1}{2}(f+g-|f-g|)$, we see that $f \wedge g\in BV(\sigma)$. The claimed bound follows by applying the triangle inequality to the identities for $f\vee g$ and $f\wedge g$.
\end{proof}

\begin{rem}
	If one is only interested in estimating the variation, then by Lemma~\ref{lem:bvc:5540} one has that
	\begin{align*}
		\myvar(f\vee g,\sigma),\myvar(f\wedge g,\sigma) \leq\myvar(f,\sigma)+\myvar(g,\sigma).
	\end{align*}
\end{rem}

\begin{rem}
	Though a lattice, $BV_\mR(\sigma)$ is not a Banach lattice under pointwise operations. Recall that for a lattice to be a Banach lattice, the norm must respect the lattice ordering. That is,  $\norm{|x|}\leq\norm{|y|}$ whenever $|x|\leq|y|$ (where $|x| = x\vee (-x)$). In $BV[0,1]$, consider the constant $f=3$ and the function $g=2-\chi_{\{0.5\}}$. Then $f> g\geq 1$ but $\norm{f} = 3 < \norm{g} = 4$.
\end{rem}

\section{Subalgebras}\label{subalgebras}

Before progressing further we shall demonstrate that $BV(\sigma)$ always contains a suitably large collections of useful functions. In particular, $BV(\sigma)$ always contains the polynomials in two real variables.

\subsection{Absolutely continuous functions}

In this section we shall consider $\sigma$ to be a subset of $\mR^2$.
Define $p_x,p_y: \sigma \to \mC$ by $p_x(x,y) = x$ and $p_y(x,y) = y$. By Theorem~\ref{ext-1var}, $p_x,p_y \in BV(\sigma)$.
 Let $\Pol$ denote the algebra of complex-valued polynomials in two variables,  $p(x,y) = \sum_{m,n=0}^N c_{mn}\, x^m y^n$. Since $BV(\sigma)$ is an algebra, the following is immediate.

\begin{lem}\label{p1-in-bv}
$\Pol$ is a subalgebra of $BV(\sigma)$.
\end{lem}

The space $\Pol$ should more correctly be denoted as $\Pol(\sigma)$, and the reader is cautioned that the same function may have distinct representations as a polynomial if the set $\sigma$ is small.
Here and throughout the paper we shall usually use the shorter form of the notation if there is no risk of confusion.

\begin{defn}
The space of \textbf{absolutely continuous functions} of $\sigma$, denoted $\AC(\sigma)$, is defined to be the closure of $\Pol$ in $(\BV(\sigma),\|\cdot\|_{BV(\sigma)}) $.
\end{defn}

Since $BV$ convergence implies uniform convergence, it is
clear that all absolutely continuous functions are in fact
continuous. From Proposition~\ref{bv-notc}, we know that when $\sigma$ is infinite and compact, $AC(\sigma)$ is a proper Banach subalgebra of $BV(\sigma)$.

In the case that $\sigma = [a,b] \subseteq \mR$, this definition agrees with the more classical one, and in that case one can readily find continuous functions of bounded variation (such as the Cantor function) which are not absolutely continuous. Moreover, it follows from Corollary~\ref{bg-bv} that $AC(\mT)$ is isomorphic to $AC_{BG}(\mT)$, the Banach algebra of absolutely continuous functions as defined by Berkson and Gillespie in \cite{BG2}.

The following simple but useful lemma is worth recording.

\begin{lem}\label{re-im}
Suppose that $\sigma$ is a nonempty compact subset of $\mR^2$. Then $f\in AC(\sigma)$ if and only if $\text{Re}(f),\text{Im}(f)\in AC(\sigma)$.
\end{lem}

\begin{proof}
This simply follows from the fact that $AC(\sigma)$ is a vector space that is closed under complex conjugation.
\end{proof}

\begin{thm}\label{ac-restrict}
Suppose that $\sigma_1,\sigma$ are nonempty compact subsets of the plane with $\sigma_1 \subseteq \sigma$. If $f \in AC(\sigma)$ then $f|\sigma_1 \in AC(\sigma_1)$, with $\norm{f|\sigma_1}_{BV(\sigma_1)} \le \norm{f}_{BV(\sigma)}$.
\end{thm}

\begin{proof}
If $f\in AC(\sigma)$, then there is a sequence of polynomials $\{p_n\}_{n=1}^\infty$ such that $p_n \to f$ as $n\to\infty$ in $BV(\sigma)$. By Theorem~\ref{bv-restrict}, $\{p_n|\sigma_1\}_{n=1}^\infty$ defines a sequence of polynomials in $AC(\sigma_1)$ such that $p_n|\sigma_1 \to f|\sigma_1$ in $BV(\sigma_1)$ and
\begin{align*}
\|f|\sigma_1\|_{BV(\sigma_1)}
  = \lim_{n\to\infty}\|p_n|\sigma_1\|_{BV(\sigma_1)}
  \leq \lim_{n\to\infty}\|p_n\|_{BV(\sigma)}
  = \|f\|_{BV(\sigma)}.
\end{align*}
\end{proof}

Let $\phi$ be an invertible affine map of the plane and let ${\hat \sigma} = \phi(\sigma)$. It is an easy exercise to check that $p \in \Pol(\sigma)$ if and only if ${\hat p} = p \circ \phi^{-1} \in \Pol({\hat \sigma})$. (We note that this is true whether one is considering real or complex affine maps.)  This immediately implies the following analogues of Theorems~\ref{aff-inv} and~\ref{ext-1var}.

\begin{thm}\label{ac-affine}
Let $\sigma$ be a nonempty compact subset of $\mR^2$ and let $\phi$ be an invertible affine map of the plane. Then $AC(\phi(\sigma))$ is isometrically isomorphic to $AC(\sigma)$.
\end{thm}

%
%

%

\begin{thm}\label{ac-ext1}
 Suppose that $\sigma$ is a nonempty compact subset of the plane with projection $\sigma_\mR$ onto the real axis. If $u \in AC(\sigma_\mR)$ and ${\hat u}$ is the extension of $u$ to $\sigma$ given by ${\hat u}(x,y) = u(x)$, then ${\hat u} \in AC(\sigma)$.
\end{thm}

\begin{proof}
Clearly if $p: \sigma_\mR \to \mC$ is a polynomial of one real variable, then its extension ${\hat p}$ lies in $\Pol$. By Theorem~\ref{ext-1var}, if $p_n \to u$ in $BV(\sigma_\mR)$ then ${\hat p}_n \to {\hat u}$ in $BV(\sigma)$.
\end{proof}

Again, one might replace the real axis here with any other line in the plane.

\noindent Some natural questions that arise are:
\begin{enumerate}
\item If $\sigma_1\subset\sigma$ and $f\in AC(\sigma_1)$, does there exist $\hat{f}\in AC(\sigma)$ such that $\hat{f}|\sigma_1=f$?
\item If $\sigma$ can be decomposed as $\sigma=\sigma_1\cup\sigma_2$ and $f:\sigma\to\mC$ such that $f|\sigma_1\in AC(\sigma_1)$ and $f|\sigma_2\in AC(\sigma_2)$, then $f\in AC(\sigma)$?
\end{enumerate}
The first question is still open in general, but we answer the second in the negative.
\begin{ex}
Let $\sigma_1=\{0\}\cup\{\frac{1}{2k-1}\}_{k=1}^\infty$, $\sigma_2=\{0\}\cup\{\frac{1}{2k}\}_{k=1}^\infty$, and $\sigma=\sigma_1\cup\sigma_2$. Define $f:\sigma\to\mC$ by
\begin{align*}
    f(x)=\begin{cases}
       0	&        \text{ if  $x=0$},\\
      (-1)^k x &	\text{ if $x=\frac{1}{k}$}.
     \end{cases}
\end{align*}

First we observe that since $f|\sigma_1(x)=-x$ and $f|\sigma_2(x)=x$, we have that $f|\sigma_1$ and $f|\sigma_2$ are polynomials and hence absolutely continuous. On the other hand, $f$ is not of bounded variation since
\begin{align*}
\myvar(f,\sigma)&	=\sum_{k=1}^\infty\left|f(\tfrac{1}{k})-f(\tfrac{1}{k+1}) \right| \\
&	=	\sum_{k=1}^\infty\left|\frac{(-1)^k}{k}-\frac{(-1)^{k+1}}{k+1} \right| \\
&	=\sum_{k=1}^\infty \frac{1}{k}+\frac{1}{k+1},
\end{align*}
which is divergent.
\end{ex}
However, if one imposes some conditions on $\sigma_1$ and $\sigma_2$, then the second question can be answered positively. We shall say that two nonempty compact sets $\sigma_1$ and $\sigma_2$ \textbf{join convexly} if for all $\vecx \in \sigma_1\setminus \sigma_2$ and $\vecy \in \sigma_2 \setminus \sigma_1$ there exists $\vecw$ on the line joining $\vecx$ and $\vecy$ with $\vecw \in \sigma_1 \cap \sigma_2$. Clearly if $\sigma_1 \cup \sigma_2$ is convex then the two sets join convexly. The following result was stated in \cite{AD3} for convex sets, although the proof only uses the  property of joining convexly. Below we present the slightly more general version along with a proof using the updated definitions of variation in the plane.

\begin{thm}\label{variation-join}
 Suppose that $\sigma_1, \sigma_2 \subseteq \mR^2$ are
nonempty compact sets which are disjoint except at their boundaries and that $\sigma_1$ and $\sigma_2$ join convexly. Let  $\sigma = \sigma_1 \cup \sigma_2$. If $f:\sigma \to \mC$,
then
  \[ \max\{\myvar(f,\sigma_1),\myvar(f,\sigma_2)\}
     \le \myvar(f,\sigma)
     \le \myvar(f,\sigma_1) + \myvar(f,\sigma_2) \]
and hence
  \[ \max\{\norm{f}_{\BV(\sigma_1)},\norm{f}_{\BV(\sigma_2)} \}
  \le  \normbv{f}
  \le \norm{f}_{\BV(\sigma_1)} + \norm{f}_{\BV(\sigma_2)}. \]
Thus, if $f|\sigma_1 \in \BV(\sigma_1)$ and $f|\sigma_2 \in
\BV(\sigma_2)$, then $f \in \BV(\sigma)$.

\end{thm}

\begin{proof}
The left hand inequalities are a simple consequence of Theorem~\ref{bv-restrict}. 

Suppose then that $S=[\vecz_0,\vecz_1,\dots,\vecz_n]$ is a finite list of points in $\sigma$.
Since $\sigma_1$ and $\sigma_2$ join convexly, for any $j$ for which $\vecz_{j-1}$ and $\vecz_{j}$ do not lie in the same $\sigma_i$, we can find a point $\vecw_j$ in $\sigma_1 \cap \sigma_2$ which lies on the line segment $\ls[\vecz_{j-1},\vecz_{j}]$. Inserting each such $\vecw_j$ into the list $S$ between $\vecz_{j-1}$ and $\vecz_{j}$ produces an extended list $S^+=[\vecz_0^+,\vecz_1^+,\dots,\vecz_{m}^+]$ with the property that for each $j$ either $\{\vecz_{j-1}^+,\vecz_j^+\} \subseteq \sigma_1$, or $\{\vecz_{j-1}^+,\vecz_j^+\} \subseteq \sigma_2$ (or both).
Note that adding these points does not change the variation factor, and can only increase the curve variation of~$f$. That is $\vf(S)=\vf(S^+)$ and $\cvar(f,S)\leq\cvar(f,S^+)$.

For $i = 1,2$, let $S_i^+=[\vecz_0^{(i)},\vecz_1^{(i)},\dots,\vecz_{n_i}^{(i)}]$ be the list obtained by removing from $S^+$ all points which do not lie in $\sigma_i$.
Every pair of consecutive points $\vecz_{j-1}^+$ and $\vecz_{j}^+$ in $S^+$ is in at least one of $S_1^+$ and $S_2^+$. Thus
\begin{align*}
\sum_{j=1}^m |f(\vecz_j^+)-f(\vecz_{j-1}^+)|	&	\leq \sum_{i=1}^2\sum_{j=1}^{n_i} \left|f\big(\vecz_j^{(i)}\big)-f\big(\vecz_{j-1}^{(i)}\big)\right|,
\end{align*}
where an empty sum is interpreted to have a value of zero. Hence
\begin{align*}
\myvar(f,S) &	=\frac{\cvar(f,S)}{\vf(S)}\\
&	\leq\frac{\cvar(f,S^+)}{\vf(S^+)}\\
&	=\frac{1}{\vf(S^+)}\sum_{j=1}^m|f(\vecz_j^+)-f(\vecz_{j-1}^+)|\\
&	\leq \sum_{i=1}^2 \frac{1}{\vf(S_i^+)} \sum_{j=1}^{n_i} \left|f\big(\vecz_j^{(i)}\big)-f\big(\vecz_{j-1}^{(i)}\big)\right|\\
&	= \frac{\cvar(f|\sigma_1,S_1^+)}{\vf(S_1^+)}+\frac{\cvar(f|\sigma_2,S_2^+)}{\vf(S_2^+)}\\
&	\leq\myvar(f|\sigma_1,\sigma_1)+\myvar(f|\sigma_2,\sigma_2).
\end{align*}
The result follows from the fact that $\|f\|_\infty$ is equal to $\|f|\sigma_1\|_\infty$ or $\|f|\sigma_2\|_\infty$.

\end{proof}

When $\sigma\subset\mR$, it follows from the classical definition of absolute continuity that $f\in AC(\sigma)$ implies $|f|\in AC(\sigma)$, and so the real-valued absolutely continuous functions $AC_\mR(\sigma)$ form a lattice. Moreover, we have seen that $\BV_\mR(\sigma)$ is a lattice under pointwise operations.
\bigskip

\noindent\textbf{Open Problem.}  For general compact $\sigma \subseteq \mR^2$, is $AC_\mR(\sigma)$ a lattice under pointwise operations?
\bigskip

The obstacle in proving that $AC_\mR(\sigma)$ is a lattice is that it is unclear whether $f\in AC(\sigma)$ implies $|f|\in AC(\sigma)$. For a positive answer, it would be sufficient to have the following result, which is certainly true in the case of $\sigma\subset\mR$.

\begin{conj}
	Suppose that $\sigma$ is a nonempty compact subset of $\mC$, $f\in AC(\sigma)$ and $\varepsilon>0$. Then there is a compact neighbourhood $U_\varepsilon$ of $Z(f)$ such that $\myvar(f, U_\varepsilon) < \varepsilon$.
\end{conj}

\subsection{Lipschitz functions}

In this section we shall consider $\sigma$ to be a subset of $\mC$. By the standard identifications,
 the identity function $\zeta: \sigma \to \mC$, $\zeta(x+iy) = x+iy = p_x(x,y)+i p_y(x,y)$ lies in $\Pol$, and so it is certainly always a function of bounded variation. Recall that the diameter of $\sigma$ is defined as $\diam{\sigma} = \max\{|\vecx - \vecy| \st \vecx,\vecy \in \sigma\}$. Let $|\sigma_\mR|$ and
$|\sigma_{i\mR}|$ denote the diameters of the projections of $\sigma$ onto the real and imaginary axes. Then, by the proof of Theorem~\ref{ext-1var}, we have that $\myvar(p_x,\sigma)\leq |\sigma_\mR|$ and $\myvar(p_y,\sigma) \leq |\sigma_{i\mR}|$, and so
  \begin{equation}\label{Vsigma-bounds}
   \max\{|\sigma_\mR|,|\sigma_{i\mR}|\}
   \leq \diam{\sigma}
   \leq \myvar(\zeta,\sigma)
   \leq \myvar(p_x,\sigma) + \myvar(p_y,\sigma)
   \leq |\sigma_\mR|+|\sigma_{i\mR}|.
  \end{equation}

\begin{defn}
Supppose that $\sigma$ is a nonempty compact subset of $\mC$. The \textbf{variation constant} of $\sigma$ is the number $C_\sigma = \myvar(\zeta,\sigma)$.
\end{defn}

The upper and lower bounds for $C_\sigma$ given in (\ref{Vsigma-bounds}) can both be obtained. Certainly if $\sigma$ lies in a line, then $C_\sigma = \diam{\sigma} = |\sigma_\mR|$ (and of course $|\sigma_{i\mR}| = 0$). On the other hand, if $\sigma = \{0,1,i,1+i\}$, Theorem~\ref{conv-bounds-thm} implies that $C_\sigma \geq 2 = |\sigma_{\mR}|+|\sigma_{i\mR}|$.


For a set $\sigma$ containing at least 2 elements, the Banach algebra $\Lip(\sigma)$ of
Lipschitz functions on $\sigma$ is the space
of all functions $f$ such that
$\norm{f}_{\Lip(\sigma)} = \norm{f}_\infty + L_\sigma(f)$ is finite,
where
  \[ L_\sigma(f) =
              \sup \Bigl\{ \frac{|f(\vecx)-f(\vecx')|}{|\vecx-\vecx'|} \st  \vecx\ne  \vecx' \in \sigma \Bigr\}. \]
(Again one might include the trivial case of a singleton set by defining $L_\sigma(f)$ to be always zero in that case.)

\begin{thm}\label{Lip-fns}
Let $\sigma$ be a nonempty compact subset of $\mC$. Then for $f \in \Lip(\sigma)$,
\begin{align*}
	\normbv{f} \le \max\{1,C_\sigma\} \norm{f}_{\Lip(\sigma)}
\end{align*}
and so $\Lip(\sigma) \subseteq BV(\sigma)$.
\end{thm}

\begin{proof}
Suppose that $f \in \Lip(\sigma)$ and that $S = \plist{\vecx_0,\dots,\vecx_n}$ is a list of points in $\sigma$. Then
  \begin{align*}
   \frac{\cvar(f,S)}{\vf(S)} & = \frac{1}{\vf(S)} \sum_{i=1}^n |f(\vecx_i) - f(\vecx_{i-1})| \\
   &\le \frac{L_\sigma(f)}{\vf(S)} \sum_{i=1}^n |\vecx_i - \vecx_{i-1}| \\
   &\le L_\sigma(f) C_\sigma.
  \end{align*}
It follows that
  \[ \normbv{f} = \norminf{f}+ \myvar(f,\sigma) \le \norminf{f} + L_\sigma(f) C_\sigma \le \max\{1,C_\sigma\} \norm{f}_{\Lip(\sigma)}. \]
\end{proof}

If $\sigma \subseteq \mR$ then every Lipschitz function is absolutely continuous and indeed $AC(\sigma)$ is the closure of $\Lip(\sigma)$ in that case. For more general sets this inclusion may fail. Before giving the example  we require the following lemma from \cite{AD1}.

\begin{lem}\label{ac-var-lim} Let $\sigma = [0,1]^2$ and for $t \in [0,1]$, let $\sigma_t = \{(x,t) \st x\in[0,1]\}$. For all $f \in AC(\sigma)$,
$\displaystyle \lim_{t \to 0^+} \myvar(f,\sigma_t) = \myvar(f,\sigma_0)$.
\end{lem}


\begin{proof}
Suppose that $f \in AC(\sigma)$. Fix $\varepsilon > 0$. By definition, there exists $p \in \Pol$ such that $\norm{f-p}_{BV(\sigma)} < \varepsilon/3$. This implies in particular that (using Lemma~\ref{lem:bvc:5540}) that for all $t$, $|\myvar(f,\sigma_t)-\myvar(p,\sigma_t)| \le \myvar(f-p,\sigma_t) \le \myvar(f-p,\sigma) < \varepsilon/3$.

Let $M = \sup\{\left|\frac{\partial^2 p}{\partial x \partial y}(x,y)\right| \st (x,y) \in \sigma\}$. By the Mean Value Theorem, there exist values $c(x,t) \in (0,1)$ such that
  \[ \frac{\partial p}{\partial x}(x,t)
   - \frac{\partial p}{\partial x}(x,0)
       = t \frac{\partial^2 p}{\partial x \partial y}(x,c(x,t)). \]
Suppose that $0 < t < \varepsilon/(3M)$.
By the Reverse Triangle Inequality then
  \begin{align*}
  |\myvar(p,\sigma_t) - \myvar(p,\sigma_0)|
  & = \left| \int_0^1
     \left| \frac{\partial p}{\partial x}(x,t) \right|
        - \left| \frac{\partial p}{\partial x}(x,0) \right|
                  \, dx \right| \\
  &\le  \int_0^1 t \left|\frac{\partial^2 p}{\partial x \partial y}(x,c(x,t)) \right| \, dx \\
  &< \frac{\varepsilon}{3M} \, M = \frac{\varepsilon}{3}.
  \end{align*}
Thus
  \begin{align*}
  |\myvar(f,\sigma_t) - \myvar(f,\sigma_0)| 
   &\le  |\myvar(f,\sigma_t) - \myvar(p,\sigma_t)|
      + |\myvar(p,\sigma_t) - \myvar(p,\sigma_0)| \\
   & \qquad\qquad   + |\myvar(p,\sigma_0) - \myvar(f,\sigma_0)| \\
      &< \varepsilon.
   \end{align*}
\end{proof}
%

The following example of a function which is in $\Lip(\sigma)$ but not in $AC(\sigma)$ was first given in \cite{Ash}.

\begin{ex}
Let $\psi: \mR \to \mR$ be the $1$-periodic function satisfying $\psi(x) = |x|$ for $-\frac{1}{2} \le x \le \frac{1}{2}$. Let $\sigma = [0,1]^2$ and for $t \in [0,1]$ let $\sigma_t = [0,1]\times\{t\}$.
For $n = 0,1,2,\dots$, let $t_n = 2^{-n}$ and define $f$ on $\sigma_{t_n}$ by $f(x,2^{-n}) = 2^{-n} \psi(2^n x)$. For each $x \in [0,1]$, now fix $f(x,t)$ to be linear in $t$ on each interval of the form $(t_{n+1},t_n)$. Finally let $f(x,0) = 0$ for all $x \in [0,1]$.

The graph of $f$ on each rectangle of the form $[k/2^m,(k+1)/2^m] \times [t_{n+1},t_n]$ is just a scaled version of graph of $f$ on $[0,1] \times [\frac{1}{2},1]$. One can use this to see that $f$ is indeed Lipschitz with $L(f) = 1$. Note however that $\myvar(f,\sigma_{t_n}) = 1$ for all $n$, while $\myvar(\sigma_0) = 0$, and so by Lemma~\ref{ac-var-lim}, $f \not\in AC(\sigma)$.
\end{ex}

Despite such examples, it is often easier to calculate $\norm{f}_{\Lip(\sigma)}$ than $\normbv{f}$. So the estimates in Theorem~\ref{Lip-fns} give a useful tool for proving convergence in $AC(\sigma)$. In particular, if $\{f_k\}$ is a sequence in $AC(\sigma)$ and $\norm{f-f_k}_{\Lip(\sigma)} \to 0$ then $f_k \to f$ in $BV(\sigma)$ and so $f \in AC(\sigma)$.

\subsection{Differentiable functions}\label{diff-fns}

One expects that sufficiently smooth functions should be absolutely continuous. Making that precise requires a degree of care since our functions are defined on compact rather than open subsets of the plane.

Even in the classical case of $\sigma = [0,1]$, there are continuous functions which fail to be of bounded variation. Indeed, it is easy to construct functions $f \in C[0,1]$ such that $f$ is infinitely differentiable on $(0,1)$ but $f$ is not of bounded variation. The function
  \[ f(x) = \begin{cases}
             x^2 \sin(x^{-2})   & x \ne 0, \\
             0                  & x =0
             \end{cases} \]
is differentiable on any open neighbourhood of $[0,1]$, but it does not lie in $BV[0,1]$.

Of course if $f'$ exists everywhere and is bounded, then $f$ is Lipschitz and hence $f \in BV[0,1]$ and indeed, in this classical setting, $f \in AC[0,1]$. We have then that $C^1[0,1] \subseteq \Lip[0,1] \subseteq AC[0,1]$.


For this section we shall consider the plane as being $\mR^2$ since we do not want to require complex differentiability of our functions.

\begin{defn} Suppose that $\sigma$ is a nonempty compact subset of $\mR^2$ and that $k$ is a positive integer. We shall denote by $C^k(\sigma)$ the vector space of all functions $f: \sigma \to \mC$ for which there exists an open neighbourhood $U$ of $\sigma$ and an extension $F$ of $f$ to $U$ such that $F\in C^k(\sigma)$.
	Similarly, we shall denote by $C^\infty(\sigma)$ the space of functions which have a $C^\infty$ extension on some open neighbourhood of $\sigma$.
\end{defn}

The reader is cautioned that there are several nonequivalent definitions of $C^k(\sigma)$  used in different areas of analysis;  see for example \cite{DD} or \cite{FLW}.

It is worth noting that if $f \in C^k(\sigma)$ then in fact $f$ admits a $C^k$ extension to the whole of $\mR^2$ and hence in particular to any open subset of the plane.

We show in \cite{DLS2} that in the general case, every $C^1$ function is absolutely continuous. This requires a significant amount of extra machinery, so here we will just show the easier result that being $C^2$ is sufficient. We begin with the case that $\sigma$ is a rectangle.

\begin{thm}\label{C2-rect}
	Suppose that $R$ is a rectangle in the plane. Then $C^2(R) \subseteq AC(R)$.
\end{thm}

\begin{proof}
	Suppose that $F \in C^2(R)$. By Lemma~\ref{re-im} and affine invariance, we may assume that $F$ is real and $R=[0,1]^2$.
	
	Fix $\varepsilon > 0$.
	Then there exist polynomials (in two variables)
	$g^{xx}$, $g^{xy}$ and $g^{yy}$ such that on the square
	\[ \norm{F_{xx} - g^{xx}}_\infty < \varepsilon, \qquad
	\norm{F_{xy} - g^{xy}}_\infty < \varepsilon, \qquad
	\norm{F_{yy} - g^{yy}}_\infty < \varepsilon. \]
	Now define $h^x: R \to \mR$ by
	\[ h^x(x,y) = F_x(0,0) + \int_0^x g^{xx}(t,0) \, dt
	+ \int_0^y g^{xy}(x,s) \, ds. \]
	Note that $h^x$ is a polynomial, and that for $(x,y) \in R$,
	\begin{align}\label{sup-dif}
		| F_x(x,y) - h^x(x,y)|
		& = \smod{ F_x(0,0) + \int_0^x F_{xx}(t,0) \, dt
			+ \int_0^y F_{xy}(x,s) \, ds - h^x(x,y)} \notag\\
		&\le \int_0^x | F_{xx}(t,0) - g^{xx}(t,0)| \, dt
		+ \int_0^y | F_{xy}(x,s) - g^{xy}(x,s)| \, ds \notag\\
		&< 2 \varepsilon.
	\end{align}
	Thus $\norm{F_x - h^x}_\infty < 2 \varepsilon$. Similarly, define
	the polynomial $h^y$ by
	\[ h^y(x,y) = f_y(0,0) + \int_0^x g^{xy}(t,y) \, dt
	+ \int_0^y g^{yy}(0,s) \, ds. \]
	A similar calculation shows that $\norm{f_y - h^y}_\infty < 2
	\varepsilon$. Note also that
	\[ \frac{\partial h^y}{\partial x} = g^{xy}. \]
	Now define a polynomial $p$ by
	\[ p(x,y) = F(0,0) + \int_0^x h^x(t,0)\, dt
	+ \int_0^y h^y(x,s) \, ds. \]
	Repeating the calculation in (\ref{sup-dif}) shows that
	$\norm{F-p}_\infty < 4 \varepsilon$. Next, for $(x,y) \in R$,
	\begin{align*}
		\frac{\partial p}{\partial x}
		&= h^x(x,0) + \int_0^y \frac{\partial h^y}{\partial x}(x,s) \, ds
		\\
		&= h^x(x,0) + \int_0^y g^{xy}(x,s) \, ds
	\end{align*}
	and
	\[ \frac{\partial p}{\partial y}
	= h^y(x,y). \]
	Thus
	\begin{align*}
		| F_x(x,y) - p_x(x,y) |
		&= \smod{ F_x(x,0) + \int_0^y F_{xy}(x,s) \, ds
			- h^x(x,0) - \int_0^y g^{xy}(x,s) \, ds } \\
		&\le | F_x(x,0) - h^x(x,0)|
		+ \int_0^y | F_{xy}(x,y) - g^{xy}(x,s)| \, ds \\
		&< 3\varepsilon
	\end{align*}
	and
	\[
	| F_y(x,y) - p_y(x,y) |
	= | F_y(x,y) - h^y(x,y) | < 2 \varepsilon.
	\]
	Given distinct $\vecx,\vecx' \in [0,1]^2$, let $\vecu = (\vecx - \vecx')/\norm{\vecx - \vecx'}$.
	Noting that $F$ was assumed to be real valued, it now follows from the Mean Value Theorem that there exists $\vecv$ on the line segment joining $\vecx$ and $\vecx'$ such that
	\[ \frac{|(F-p)(\vecx) - (F-p)(\vecx')|}{\norm{\vecx - \vecx'}}
	= |\nabla (F-p)(\vecv) \cdot \vecu| \le \sqrt{13} \varepsilon \]
	and hence
	$\norm{F-p}_{\Lip(R)} < (4 + \sqrt{13}) \varepsilon$.
	
	We can therefore find a sequence $p_n$ of polynomials such that
	\[ \lim_{n \to \infty} \norm{F-p_n}_{\Lip(R)} = 0\]
	and hence by Theorem~\ref{Lip-fns}, $F \in \AC(R)$.
\end{proof}

It is worth noting that it is vital in this result that $f$ be differentiable on the boundary of $R$, not just on the interior.

\begin{thm}
	If $\sigma\subset\mC$ is nonempty and compact, then $C^2(\sigma)\subseteq AC(\sigma)$.
\end{thm}

\begin{proof}
	Suppose that $f\in C^2(\sigma)$. By definition, $f$ has an extension to a $C^2$ function defined on some open neighbourhood of $\sigma$. Taylor's theorem implies that $f$ satisfies the hypotheses of the Whitney Extension Theorem \cite{W} and hence $f$ may be extended to a $C^2$ function $F$ on $\mR^2$.
	
	Let $R$ denote a rectangle whose interior contains $\sigma$. Then the previous theorem tells us that $F|R\in AC(R)$, and hence $f=F|\sigma\in AC(\sigma)$.
\end{proof}
Since the polynomials are elements of $C^2(\sigma)$, we have the following corollary.
\begin{cor}
	The space $C^2(\sigma)$ is dense in $AC(\sigma)$.
\end{cor}

We note that not every absolutely continuous function is differentiable. In \cite{DLS2} we show that any function which is continuous and piecewise planar is absolutely continuous.

\section{Absolute continuity is a local property}

In this section we shall show that being in $\AC(\sigma)$ is determined by the behaviour of the function in a neighbourhood of each point.

Recall that the classical definition of absolute continuity for $f: [0,1] \to~\mC$ is that for all $\varepsilon > 0$ there exists $\delta > 0$ such that if $\{[x_i,x_i']\}_{i=1}^n$ is any collection of disjoint subintervals of total length less than $\delta$, then $\sum_{i=1}^n |f(x_i)-f(x_i')| < \varepsilon$.

Suppose that $f: [0,1] \to \mC$, and that for all $x \in [0,1]$ there exists an open interval $V_x$ containing $x$ such that $f|\overline{V}_x \in AC(\overline{V}_x)$. By compactness we can cover $[0,1]$ with a finite collection of these intervals $V_{x_1},\dots,V_{x_m}$. It is then easy to show that if $f$ satisfies the classical definition on each interval $\overline{V}_x$ then it must satisfy it on the whole interval $[0,1]$. Conversely, if $f \not\in AC[0,1]$ there must be some point $x_0 \in [0,1]$ such that $f$ is not absolutely continuous on the closure of any open interval containing $x_0$. The aim of this section is to show that this property extends to the new definition of absolute continuity given in this paper.

We shall say that a set $U$ is a \textit{compact neighbourhood} of a point $\vecx \in \sigma$ (with respect to $\sigma$) if there exists an open set $V\subseteq\mC$ containing $\vecx$ such that $U = \sigma \cap \overline{V}$.

\begin{thm}\label{patching-lemma}[Patching Lemma]
	Suppose that $\sigma$ is a nonempty compact subset of $\mC$ and $f : \sigma \to \mC$. Then $f \in \AC(\sigma)$ if and only if for every point $\vecx \in \sigma$ there exists a compact neighbourhood $U_\vecx$ of $\vecx$ in $\sigma$ such that $f|U_\vecx \in \AC(U_\vecx)$.
\end{thm}

The main step in proving the Patching Lemma is the following extension result. As usual we shall let $\interior{R}$ denote the interior of a set $R$, and let $\supp g$ denote the support of a function $g$.

\begin{lem}\label{ext-gen} Let $\sigma$ be a nonempty compact subset of $\mC$.
	Suppose that $R$ is a closed rectangle in $\mC$ with $\sigma_R=R \cap \sigma \ne \emptyset$.
	Suppose that $g \in \AC(\sigma_R)$ with $\supp g \subset \interior{R}$. Then the function
	\[ \tilde g(\vecx) = \begin{cases}
		g(\vecx) & \textup{ if $\vecx \in \sigma_R$},\\
		0          &  \textup{ if  $\vecx \in \sigma \setminus \sigma_R$}
	\end{cases}
	\]
	lies in $\AC(\sigma)$.
\end{lem}

\begin{figure}[ht!]
	\begin{center}
		\begin{tikzpicture}[scale=1]
			%
			%
            \draw[fill,blue!10] (2,2.5) circle (0.6cm);
			\draw[blue,thick] (2,2.5) circle (0.6cm);
			\draw[fill,blue!10] (4.5,2) ellipse (0.8cm and 3cm);
			\draw[blue,thick] (4.5,2) ellipse (0.8cm and 3cm);
			\draw[red] (1,-0.5) -- (6,-0.5) -- (6,4.5)  -- (1,4.5) -- (1,-0.5);

			\fill[gray] (2,2.5) circle (0.4cm);
			\draw[black] (2,2.5) circle (0.4cm);
			\fill[gray] (4.5,2.7) ellipse (0.6cm and 0.8cm);
			\draw[black] (4.5,2.7) ellipse (0.6cm and 0.8cm);
			\fill[gray] (4.7,1.2) ellipse (0.2cm and 0.25cm);
			\draw[black] (4.7,1.2) ellipse (0.2cm and 0.25cm);
			
			\draw[black] (1.3,0) node[above] {$R$};
			\draw[black] (3,1.7) node {$\sigma$};
			\draw[black] (3,3.4) node {$\supp g$};
			
			\draw[black] (2.8,1.8) -- (2.3,2.1);
			\draw[black] (3.2,1.7) -- (3.8,1.7);
			
			\draw[black] (2.2,2.7) -- (2.6,3.2);
			\draw[black] (3.6,3.4) -- (4.1,3);
		\end{tikzpicture}
	\end{center}
	\caption{The setting for Lemma~\ref{ext-gen}.}
\end{figure}
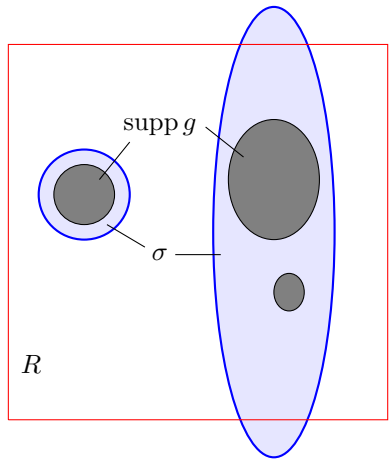


\begin{proof}
	By the affine invariance of absolute continuity, it suffices to consider the case that  $R = [0,1] \times [0,1]$. Suppose that $g\in AC(\sigma_R)$ satisfies supp$(g)\subset\text{int}(R)$.
	Choose a closed square $R_0 = [a,b] \times [a,b]$ (with sides parallel to those of $R$) such that $\supp g \subseteq R_0 \subseteq \interior{R}$.  Let $\chi_1: [0,1] \to [0,1]$ be a $C^\infty$ bump function which is zero at and near the endpoints, 1 on an open neighbourhood of $[a,b]$ and monotonic on $[0,a]$ and $[b,1]$. Let $\chi: R \to [0,1]$, $\chi(x,y) = \chi_1(x) \chi_1(y)$. Then $\chi \in C^\infty(R)$ and, using  Theorem~\ref{ext-1var},
	\begin{align*}
		\norm{\chi}_{\BV(R)}\leq 
                      \norm{\chi_1}_{BV[0,1]}^2 = 9.
	\end{align*}
	
	

	Fix $\varepsilon > 0$.
	Since $g \in \AC(\sigma_R)$, there exists a polynomial $p \in C^\infty(\sigma_R)$ such that $9\norm{g-p}_{\BV(\sigma_R)} < \varepsilon$. Clearly  $\chi p \in C^\infty(\sigma_R)$ and
	\[ \norm{g - \chi p }_{\BV(\sigma_R)} = \norm{\chi(g - p)}_{\BV(\sigma_R)} \le \norm{\chi}_{\BV(\sigma_R)} \norm{g-p}_{\BV(\sigma_R)} < \varepsilon.\]
	Let $\tilde p$ denote the extension of $\chi p$ to $\sigma$ determined by setting
	\[ {\tilde p}(\vecx) = \begin{cases}
		\chi p(\vecx)      & \text{$\vecx \in \sigma_R$,}\\
		0                     & \text{$\vecx \in \sigma \setminus \sigma_R$.}
	\end{cases}
	\]
	Note that $\tilde p \in C^\infty(\sigma) \subseteq \AC(\sigma)$. Let $h = \tilde g -\tilde p$. It will be shown that $\norm{h }_{\BV({\sigma})} < 5 \varepsilon$.
	
	Let $S = [\vecx_0,\vecx_1,\dots,\vecx_m]$ be an ordered list of elements in $\sigma$.
	Partition the indices $1,\dots,m$ into
	\begin{align*}
		J_1 &= \{j \st \hbox{$\vecx_{j-1},\vecx_j \in \sigma_R$}\}, \\
		J_2 &= \{j \st \hbox{$\vecx_{j-1},\vecx_j \not\in \sigma_R$}\}, \\
		J_3 &= \{1,\dots,m\} \setminus (J_1 \cup J_2).
	\end{align*}
Then (interpreting empty sums as being zero)
	\begin{align}\label{sum-of-diffs}
		\sum_{j=1}^m | h(\vecx_j) - h(\vecx_{j-1}) |
		&= \sum_{i=1}^3 \sum_{j\in J_i} | h(\vecx_j) - h(\vecx_{j-1}) |
                          \notag\\
		&\le \sum_{j\in J_1} | h(\vecx_j) - h(\vecx_{j-1}) | + 0
		+ |J_3| \norm{h}_\infty.
	\end{align}
	
	\begin{figure}[ht!]
		\begin{center}
			\begin{tikzpicture}[scale=1]
				%
				%
        \draw[fill,blue!10] (2,2.5) circle (0.6cm);
		\draw[blue,thick] (2,2.5) circle (0.6cm);
		\draw[fill,blue!10] (4.5,2) ellipse (0.8cm and 3cm);
	    \draw[blue,thick] (4.5,2) ellipse (0.8cm and 3cm);
		\draw[red] (1,0) -- (6,0) -- (6,4)  -- (1,4) -- (1,0);
				
				\path node (z0) at (1.7,2.8) [circle,draw,fill=black!50,inner sep=0pt, minimum width=4pt] {}
				node (z1) at (2.3,2.3) [circle,draw,fill=black!50,inner sep=0pt, minimum width=4pt] {}
				node (z2) at (4.3,4.6) [circle,draw,fill=black!50,inner sep=0pt, minimum width=4pt] {}
				node (z3) at (4.7,4.2) [circle,draw,fill=black!50,inner sep=0pt, minimum width=4pt] {}
				node (z4) at (4,3) [circle,draw,fill=black!50,inner sep=0pt, minimum width=4pt] {}
				node (z5) at (5,3) [circle,draw,fill=black!50,inner sep=0pt, minimum width=4pt] {}
				node (z6) at (4.8,-0.4) [circle,draw,fill=black!50,inner sep=0pt, minimum width=4pt] {}
				node (z7) at (4.2,-0.4) [circle,draw,fill=black!50,inner sep=0pt, minimum width=4pt] {}
				node (z8) at (3.9,1) [circle,draw,fill=black!50,inner sep=0pt, minimum width=4pt] {};
				\draw[black,thick] (z0) -- (z1) -- (z2) -- (z3) -- (z4) -- (z5) -- (z6) -- (z7) -- (z8);
				\draw[black] (z0) node[right] {$\vecx_0$};
				\draw[black] (z1) node[left] {$\vecx_1$};
				\draw[black] (z2) node[right] {$\vecx_2$};
				\draw[black] (z3) node[left] {$\vecx_3$};
				\draw[black] (z4) node[below] {$\vecx_4$};
				\draw[black] (z5) node[above] {$\vecx_5$};
				\draw[black] (z6) node[right] {$\vecx_6$};
				\draw[black] (z7) node[left] {$\vecx_7$};
				\draw[black] (z8) node[right] {$\vecx_8$};
				
				\draw[black] (1.3,0) node[above] {$R$};
				\draw[black] (3,1.7) node {$\sigma$};
				
				\draw[black] (2.8,1.8) -- (2.3,2.1);
				\draw[black] (3.2,1.7) -- (3.8,1.7);
			\end{tikzpicture}
		\end{center}
		\caption{In this example $J_1 = \{1,5\}$, $J_2 = \{3,7\}$, $J_3 = \{2,4,6,8\}$ and $S_R = [\vecx_0,\vecx_1,\vecx_4,\vecx_5]$.}
	\end{figure}
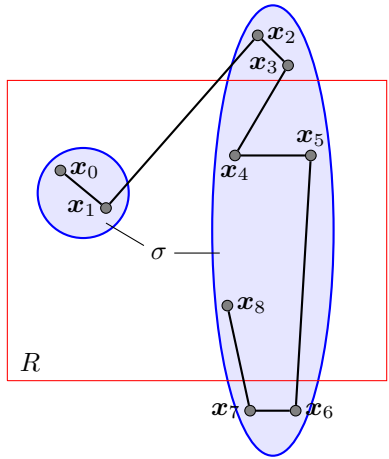
	
	Suppose first that $J_1 \ne \emptyset$. Since $\sum_{j\in J_1} |h(\vecx_j)-h(\vecx_{j-1})|$ is not necessarily summing over a list, we will add extra terms so that the resulting sum is over a list. Let $S_R = [\vecx_{j_0},\dots,\vecx_{j_\ell}]$ be the sublist of $S$ obtained by omitting points that are not in $\sigma_R$.
	Note that $S_R$ is a nonempty list of points in $\sigma_R$. Thus
	\begin{align*}
		\sum_{j \in J_1} | h(\vecx_j) - h(\vecx_{j-1}) |
		& \le \sum_{i = 1}^\ell | h(\vecx_{j_i}) - h(\vecx_{j_{i-1}})| = \cvar(h,S_R).
	\end{align*}
	Of course if $J_1 = \emptyset$ then this estimate holds trivially.
	
	We now need to deal with the edges corresponding to $J_3$.
	Note that each element of $J_3$ corresponds to $\gamma_S$ crossing one of the four edge lines of $R$. It follows that $|J_3| \le 4 \vf(S)$. Substituting this into Equation~\ref{sum-of-diffs}  with $\vf(S_R) \le \vf(S)$ (Lemma~\ref{vf-lemma}) gives that
	\clearpage
	\begin{align*}
		 \frac{1}{\vf(S)}\sum_{j=1}^m | h(\vecx_j) - h(\vecx_{j-1}) |
		& \le \frac{\cvar(h,S_R)}{\vf(S_R)} + 4 \norm{h}_\infty \\
		& \le 4 \norm{h}_{\BV(\sigma_R)} \\
		& < 4 \varepsilon.
	\end{align*}
	Taking the supremum over all ordered lists $S$ and adding another $\norm{h}_\infty$ shows that $\norm{h}_{BV(\sigma)} < 5 \varepsilon$. Since $h=\tilde{g}-\tilde{p}$, it follows that $\tilde g \in \AC(\sigma)$.
\end{proof}

\medskip\noindent
\textit{Proof of Theorem \ref{patching-lemma}.}
The forward implication is obvious.

For the reverse implication, for each $\vecx\in\sigma$, choose a compact neighbourhood
$U_\vecx = \mathrm{cl}(V_\vecx) \cap \sigma$ such that $f|U_\vecx \in \AC(U_\vecx)$. One may, by taking a further restriction, assume that each $V_\vecx$ is a rectangle.

By compactness  we can choose a finite open subcover $V_1,\dots,V_m$
of $\sigma$. Now choose $C^\infty$ functions $\chi_1,\dots,\chi_m:
\mR^2 \to [0,1]$ such that $\supp \chi_j \subseteq V_j$ for each $j$
and $\sum_{j=1}^m \chi_j = 1$ on $\sigma$. Let $f_j = \chi_j f$ so that $f_j|_{U_{\vecx}}\in AC(U_\vecx)$. Then $\supp f_j \subseteq V_j$ and hence, by
Lemma~\ref{ext-gen},
$f_j$ lies in $\AC(\sigma)$.
Since  $\sum_{j=1}^m
 f_j = f$ we have $f \in \AC(\sigma)$.
\hfill\qed

\begin{cor}\label{disjoint}
	Suppose that $\sigma_1$ and $\sigma_2$ are disjoint nonempty compact sets in the plane, that $\sigma = \sigma_1 \cup \sigma_2$ and that  $f: \sigma \to \mC$. If $f|\sigma_1 \in \AC(\sigma_1)$ and $f|\sigma_2 \in \AC(\sigma_2)$ then $f \in \AC(\sigma)$.
\end{cor}

\section{Additional Properties of $AC(\sigma)$}

The following collection of results are useful in spectral theory contexts, such as the theory of $\mathfrak{U}$-spectral operators in \cite{CF}. To apply that theory, one would like to know that $\AC(\sigma)$ is topologically admissible and inverse closed.

\begin{defn}
	Let $\Omega$ be a nonempty closed subset of the complex plane. An algebra $\mathfrak{U}$ of complex-valued functions defined on $\Omega$ is \textbf{normal} if for every finite open cover $\{U_i\}_{i=1}^n$ of $\Omega$, there exists $\{f_i\}_{i=1}^n\subseteq \mathfrak{U}$ such that for all $i$
	\begin{enumerate}
		\item $f_i(\Omega) \subseteq [0,1]$;
		\item $\supp(f_i) \subseteq U_i$;
		\item $\sum_{i=1}^n f_i = 1$ on $\Omega$.
	\end{enumerate}
	$\mathfrak{U}$ is \textbf{admissible} if it is normal and
	\begin{enumerate}
		\item the constant function 1 and the identity function $\zeta$ are in $\mathfrak{U}$;
		\item for every $f\in\mathfrak{U}$ and every $\xi\notin\supp(f)$, the function $f_\xi$ defined by
		\begin{align*}
			f_\xi(z) = \begin{cases}
				\frac{f(z)}{\xi - z} & \text{ if $ z\neq\xi$},\\
				0 & \text{ if $z=\xi$}
			\end{cases}
		\end{align*}
		is in $\mathfrak{U}$.
	\end{enumerate}
	If $\mathfrak{U}$ is a subalgebra of $C(\Omega)$, then $\mathfrak{U}$ is \textbf{topologically admissible} if it is admissible and
	\begin{enumerate}
		\item $\mathfrak{U}$ is equipped with a locally convex topology $\tau$ such that whenever $(f_n)_{n\in\mN}$ is a Cauchy sequence in $\tau$ and converges pointwise to 0, then $f_n\to0$ in $\tau$;
		\item for every $f\in \AC(\sigma)$, the map $\xi\mapsto f_\xi$ is continuous on $\mC\backslash\supp(f)$.
	\end{enumerate}
	
	Lastly, $\mathfrak{U}$ is \textbf{inverse closed} if whenever $f\in\mathfrak{U}$ and $|f|>0$, then $\frac{1}{f}\in\mathfrak{U}$.
\end{defn}

\begin{rem}
For $\xi\notin\supp(f)$, it may be helpful to think of $f_\xi$ being defined by
\begin{align*}
			f_\xi(z) = \begin{cases}
				\frac{f(z)}{\xi - z} & \text{ if $z\in\supp(f)$},\\
				0 & \text{ if $z\notin\supp(f)$}.
			\end{cases}
		\end{align*}
\end{rem}

\begin{prop}\label{inv-closed}
	$\AC(\sigma)$ is inverse closed.
\end{prop}

\begin{proof}
	Suppose that $f\in \AC(\sigma)$ satisfies $|f|>0$, so that $\delta:=\min\limits_{z\in\sigma}|f(z)|>0$. Let $(p_n)_{n\in\mN}$ be a sequence of polynomials converging to $f$ in $BV(\sigma)$. Without loss, we may assume that $2|p_n|>\delta$, and hence $\frac{1}{p_n}\in C^\infty(\sigma)$, for all $n$. Firstly, we have that
	\begin{align*}
		\norm{\tfrac{1}{p_n}-\tfrac{1}{f}}_\infty = \norm{\tfrac{p_n-f}{p_nf}}_\infty \leq \tfrac{2}{\delta^2}\|p_n-f\|_\infty \to 0.
	\end{align*}
Also, the quantities $\norm{p_n}_\infty$ and $\myvar(p_n,\sigma)$ are uniformly bounded in $n$.
	Now given a list of points in $\sigma$, say $S=[z_j]_{j=0}^m$, we have that
	\begin{align*}
  & \; \frac{1}{\vf(S)}  \cvar\left(\tfrac{p_n-f}{p_nf},S\right)\\
	&\leq  \frac{4}{\delta^4\vf(S)} \sum_{j=1}^m\Big| p_n(z_{j-1})f(z_{j-1})\big[ p_n(z_j)-f(z_j) \big]\\*
	& \hskip 3cm
        - p_n(z_{j})f(z_j)\big[ p_n(z_{j-1}) - f(z_{j-1}) \big] \Big|\\
	& = 	\frac{4}{\delta^4\vf(S)} \sum_{j=1}^m \Big| f(z_{j-1})\big[ p_n(z_{j-1})-p_n(z_j) \big] \big[ p_n(z_j)-f(z_j) \big]\\*
	& \hskip 2.5cm
         + p_n(z_j)\big[ f(z_{j-1})-f(z_j) \big] \big[ p_n(z_{j-1})-f(z_{j-1}) \big]\\*
	&\hskip 2.5cm \qquad  + p_n(z_j)f(z_{j-1})\big[ p_n(z_j)-f(z_j)-p_n(z_{j-1})+f(z_{j-1}) \big] \Big|\\
	&\leq \frac{4}{\delta^4}\Big( \norm{f}_\infty \norm{p_n - f}_\infty \myvar(p_n,\sigma)
                     + \norm{p_n}_\infty \norm{p_n-f}_\infty \myvar(f,\sigma) \\*
	&\hskip 5.5cm
                         +\norm{p_n}_\infty \norm{f}_\infty \myvar(p_n-f,\sigma) \Bigr) \\*
 	&	\to 0.
	\end{align*}
	Thus $\frac{1}{f}\in \AC(\sigma)$.
\end{proof}

From this result we have some simple corollaries.

\begin{cor}
Suppose that $f,g \in \AC(\sigma)$ and that $g \ne 0$ on $\sigma$. Then $\frac{f}{g} \in \AC(\sigma)$.
\end{cor}

\begin{cor}
	If $\lambda\in\sigma$ is an isolated point, then the function $g_\lambda:\sigma\to\mC$ defined by
	\begin{align*}
		g_\lambda(z) = \begin{cases}
			0 & \qquad \text{if } z=\lambda,\\
			(z-\lambda)^{-1} &\qquad \text{if } z\neq\lambda,
		\end{cases}
	\end{align*}
	is in $AC(\sigma)$.
\end{cor}

\begin{proof}
	Since $\lambda\in\sigma$ is isolated, then $\chi_{\{\lambda\}}$ is in $\AC(\sigma)$. Again denoting the identity function on $\sigma$ by $\zeta$, the function $f_\lambda=\zeta-\lambda + \chi_{\{\lambda\}}$ is in $\AC(\sigma)$ and satisfies $|f_\lambda|>0$. By Proposition~\ref{inv-closed}, $\frac{1}{f_\lambda}\in \AC(\sigma)$, and so $g_\lambda = \frac{1}{f_\lambda}-\chi_{\{\lambda\}} \in AC(\sigma)$.
\end{proof}

\begin{thm}
	If $\sigma$ is a nonempty compact subset of $\mR^2$, then $\AC(\sigma)$ is topologically admissible and inverse closed.
\end{thm}

\begin{proof}
	Firstly, $\AC(\sigma)$ is normal since $C^\infty(\sigma)\subseteq \AC(\sigma)$ is normal, and we have already seen that $\AC(\sigma)$ is inverse closed, so all that remains is to show that $\AC(\sigma)$ is topologically admissible. So suppose that $f\in \AC(\sigma)$ and $\xi \in \mC\backslash\supp(f)$. Then there is an open set $U\supseteq \supp(f)$ disjoint from some neighbourhood of $\xi$ and a $g\in C^\infty(\mR^2)$ such that $g|_{\supp(f)}=1$ and $g|_{U^c}=0$. Then $g_\xi$ defined by
	\begin{align*}
		g_\xi(z) = \begin{cases}
			0 & \qquad \text{ if } z\in U^c,\\
			\frac{g(z)}{\xi -z} &\qquad \text{ if } z\in U
		\end{cases}
	\end{align*}
	is in $C^\infty(\mR^2)$ and $fg_\xi = f_\xi$ on $\sigma$. Thus $f_\xi\in \AC(\sigma)$ and $\AC(\sigma)$ is admissible.
	
	As $\AC(\sigma)$ is complete and the norm topology on $\AC(\sigma)$ is stronger than the topology of pointwise convergence, any Cauchy sequence that converges pointwise to 0 must converge to 0 in $\|\cdot\|_{BV(\sigma)}$. So it remains to check that $\xi\mapsto f_\xi$ is continuous on $\mC\backslash\supp(f)$ for all $f\in \AC(\sigma)$. Of course this is true when $f=0$, so we assume that $f \neq 0$ from now.
	
	We will denote by $\zeta$ the identity function on $\sigma$, by $M_\sigma$ the maximum of $1$ and $\|\zeta\|_{BV(\sigma)}$, and by $s$ the square function $s(z)=z^2$. Let $ \varepsilon>0$ and $\xi\in\mC\backslash\supp(f)$. For $w\in\mC$, define $d_w = \inf\limits_{z\in\supp(f)}|w-z|$ and let $\eta\in \mC\backslash\supp(f)$ satisfy
	\begin{align*}
		|\xi-\eta|< \min \left\{ 1, \frac{ d_\xi}{2},  \frac{d_\xi^2\varepsilon}{4 \|f\|_{BV(\sigma)}},  \frac{d_\xi^4\varepsilon}{8 \norm{f}_{BV(\sigma)} \big( (|\xi|^2 + 3|\xi|+1) M_\sigma + \|s\|_{BV(\sigma)} \big)} \right\}.
	\end{align*}
	Noting that $2d_\eta \geq d_\xi$ (as $2|\xi-\eta| < d_\xi$), we have that
	\begin{align} \label{sup-est}
		\|f_\xi-f_\eta\|_\infty \leq \|f\|_\infty \sup_{z\in\supp(f)}\left| \frac{\xi-\eta}{(\xi-z)(\eta-z)} \right| \leq \frac{2\|f\|_\infty |\xi-\eta|}{d_\xi^2}.
	\end{align}
	Moreover, suppose $S=[z_j]_{j=0}^n$ is a finite sequence in $\sigma$ and define
	\begin{align*}
		J_1 &	= \{j\,:\, z_j\in\supp(f),z_{j-1}\notin\supp(f) \}\\
		J_2 &	= \{j\,:\, z_{j-1}\in\supp(f),z_j\notin\supp(f) \}\\
		J_3 &	= \{1,2,\dots,n\}\backslash(J_1\cup J_2).
	\end{align*}
Then we have that
\begin{align*}
		&\frac{\cvar(f_\xi - f_\eta , S)}{\vf(S)} \\
 =\, & \frac{1}{\vf(S)} \sum_{j\in J_1}\left|\frac{ f(z_j) (\xi-\eta) }{ (\xi-z_j) (\eta-z_j) }\right| + \frac{1}{\vf(S)} \sum_{j\in J_2}\left|\frac{ f(z_{j-1}) (\xi-\eta) }{ (\xi-z_{j-1}) (\eta-z_{j-1}) }\right| \\
		& \hspace*{.5cm} +\frac{1}{\vf(S)} \sum_{j\in J_3}\left| f(z_j)\left[ \frac{\xi-\eta}{(\xi-z_j)(\eta-z_j)}\right] - f(z_{j-1}) \left[ \frac{\xi-\eta}{(\xi-z_{j-1})(\eta-z_{j-1})} \right] \right|\\
\, =& \frac{|\xi-\eta|}{\vf(S)} \sum_{j\in J_1}\left|\frac{ f(z_j) - f(z_{j-1}) }{ (\xi-z_j) (\eta-z_j) }\right| + \frac{|\xi-\eta|}{\vf(S)} \sum_{j\in J_2}\left|\frac{ f(z_j)-f(z_{j-1}) }{ (\xi-z_{j-1}) (\eta-z_{j-1}) }\right|\\
		&\hspace*{.5cm}+  \frac{|\xi-\eta|}{\vf(S)} \sum_{j\in J_3} \left| \frac{ f(z_j)(\xi-z_{j-1})(\eta-z_{j-1}) - f(z_{j-1}) (\xi-z_j)(\eta-z_j) }{ (\xi-z_j)(\eta-z_j)(\xi-z_{j-1})(\eta-z_{j-1}) } \right|\\
\, \leq &	\frac{2|\xi-\eta|\myvar(f,\sigma)}{d_\xi^2} + \frac{4|\xi-\eta|}{d_\xi^4\vf(S)} \sum_{j\in J_3} \Big| \xi\eta\big[ f(z_j)-f(z_{j-1}) \big] \\
		&\hspace*{.5cm} - (\xi+\eta)\big[ z_{j-1}f(z_j) - z_jf(z_{j-1}) \big] + z_{j-1}^2f(z_j) - z_j^2f(z_{j-1})  \Big|\\
		\leq \, & \frac{4|\xi-\eta|}{d_\xi^4} \Big( |\xi\eta|\myvar(f,\sigma)  + |\xi+\eta| \big( \|\zeta\|_{\BV(\sigma)}  \|f\|_{\BV(\sigma)} \big) + \|s\|_{\BV(\sigma)}\|f\|_{\BV(\sigma)}\Big) \\
		&\hspace*{.5cm} + \frac{2|\xi-\eta|\myvar(f,\sigma)}{d_\xi^2}.
\end{align*}
Recalling that $|\xi-\eta|<1$, we then have that
\begin{align*}
		&\frac{\cvar(f_\xi - f_\eta , S)}{\vf(S)}< \, \frac{4 |\xi-\eta|}{d_\xi^4} \norm{f}_{BV(\sigma)} \left( |\xi|^2 + |\xi| + \left( 2|\xi|+1 \right) \|\zeta\|_{BV(\sigma)} + \|s\|_{BV(\sigma)} \right)\\
		&\hspace*{3.5cm} + \frac{2|\xi-\eta|\myvar(f,\sigma)}{d_\xi^2} \\
		& \qquad \leq\frac{4|\xi-\eta|}{d_\xi^4} \norm{f}_{BV(\sigma)}\big( (|\xi|^2+3|\xi|+1) M_\sigma + \|s\|_{BV(\sigma)} \big) + \frac{2|\xi-\eta|\myvar(f,\sigma)}{d_\xi^2} \\
		& \qquad <  \frac{\varepsilon}{2} + \frac{2|\xi-\eta|\myvar(f,\sigma)}{d_\xi^2}.
	\end{align*}
	Taking the supremum over all $S$ then yields
 	\begin{align*}
		\myvar(f_\xi-f_\eta,\sigma) \leq \frac{\varepsilon}{2} + \frac{2|\xi-\eta|\myvar(f,\sigma)}{d_\xi^2}.
	\end{align*}
	Combining this with \eqref{sup-est} we have that
	\begin{align*}
		\|f_\xi-f_\eta\|_{BV(\sigma)} \leq \frac{\varepsilon}{2} + \frac{2|\xi-\eta| \normbv{f}}{d_\xi^2} < \varepsilon.
	\end{align*}
	So $\xi\mapsto f_\xi$ is continuous on $\mC\backslash\supp(f)$ for all $f\in \AC(\sigma)$ and thus $\AC(\sigma)$ is topologically admissible.
\end{proof}

\section*{Acknowledgements}
The work of the third author was supported by the Research Training Program of the Department of Education and Training of the Australian Government.

%
%
\bibliographystyle{amsalpha}

\end{document}